\documentclass[reqno]{amsart}
\usepackage{geometry}
\usepackage{amssymb}
\usepackage{amsmath,amsthm}
\usepackage{amsmath,amscd}
\usepackage{amsfonts}
\usepackage{calrsfs}
\usepackage{pstricks}
\usepackage{epsf}
\usepackage{xypic}
\usepackage{graphicx}
\usepackage{extarrows}
\usepackage{chemarrow}
\usepackage{xcolor}
\usepackage[colorlinks=true,linkcolor=blue,citecolor=blue]{hyperref}
\usepackage{amsmath}
\usepackage[noadjust]{cite}
\newtheorem{theorem}{Theorem}[section]
\newtheorem{lemma}[theorem]{Lemma}
\newtheorem{corollary}[theorem]{Corollary}
\newtheorem{proposition}[theorem]{Proposition}
\theoremstyle{definition}
\newtheorem{definition}[theorem]{Definition}
\newtheorem{example}[theorem]{Example}

\theoremstyle{remark}
\newtheorem{remark}[theorem]{Remark}

\numberwithin{equation}{section}

\DeclareSymbolFont{EulerExtension}{U}{euex}{m}{n}
\DeclareMathSymbol{\euintop}{\mathop} {EulerExtension}{"52}
\DeclareMathSymbol{\euointop}{\mathop} {EulerExtension}{"48}

\def\A{\mathcal{A}}

\def\M{\mathcal{M}}

\def\I{\mathcal{I}}
\def\E{\mathbb{E}}
\def\B{\mathcal{B}}
\def\C{\mathcal{C}}
\def\s{\mathfrak{s}}

\def\del{\delta}
\def\dr{\ar@{->}[r]}

\def\Y{\mathcal{Y}}

\def\add{\mbox{add}}
\def\Ext{\mbox{Ext}}
\def\Hom{\mbox{Hom}}

\def\proj{\mbox{Proj}}
\def\inj{\mbox{Inj}}
\def\Ker{\mbox{Ker}\hspace{.01in}}
\def\Im{\mbox{Im}\hspace{.01in}}

\begin{document}

\title{$\s$-Recollements and its localization}

\author{Yonggang Hu}
\address{Department of Mathematical Sciences, Tsinghua University, 100084 Beijing,  P. R. China }
\email{huyonggang@emails.bjut.edu.cn}
\thanks{Yonggang Hu  was supported by National Natural Science Foundation of China (Grant  Nos. 11671126, 12071120). Panyue Zhou was supported
	by the National Natural Science Foundation of China (Grant No. 11901190) and by the Scientific Research Fund of Hunan Provincial Education Department (Grant No. 19B239).}

\author{Panyue Zhou}
\address{College of Mathematics, Hunan Institute of Science and
Technology, 414006 Yueyang, Hunan,
	P. R. China}
\email{panyuezhou@163.com}
\thanks{}





\begin{abstract}
We introduce a new concept of $\s$-recollements of extriangulated categories, which generalizes
recollements of abelian categories, recollements of triangulated categories, as well as recollements of extriangulated categories. Moreover, some basic properties of $\s$-recollements are presented. We also discuss  the behavior of the localization theory on the adjoint pair of exact functors. Finally, we provide a method to obtain a recollement of triangulated categories from $\s$-recollements of extriangulated categories via the  localization theory.

\end{abstract}

\maketitle

\section{Introduction}
The recollement of triangulated categories was introduced first by Beilinson, Bernstein, and Deligne, see \cite{BBD}. It is an important tool in algebraic geometry \cite{BBD} and representation theory \cite{CPS,PS}.  A fundamental example of a recollement  of abelian categories appeared in the construction of perverse sheaves by MacPherson and Vilonen \cite{MV}. A recollement of triangulated (or, abelian,  extriangulated) categories, see the references \cite{BBD,MV,WWZ}, is  a
diagram of functors between triangulated (or, abelian,  extriangulated) categories of the following shape (\ref{dia0-0}), which satisfies some assumptions.

\begin{equation}\label{dia0-0}
  \xymatrix{\mathcal{A}\ar[rr]&&\ar@/_1pc/[ll]\ar@/^1pc/[ll]\mathcal{B}
\ar[rr]&&\ar@/_1pc/[ll]\ar@/^1pc/[ll]\mathcal{C}}
\end{equation}
\vskip 4mm
From the diagram (\ref{dia0-0}), recollements of  triangulated (abelian) categories can be viewed as `exact sequences' of triangulated (abelian) categories, which describe the middle term by a subcategory and a quotient category. It should be noted that recollements of abelian or triangulated categories are used widely in studying representation theory and homological theory. For instance, pioneer researchers have studied connections between recollements of triangulated categories in connection with tilting theory\cite{LAKL,BaP,K,Y,CX,CX01,CX03}, homological conjectures \cite{CX02,Hap,QH}, Hochschild theory \cite{HY} and homological dimensions \cite{CX02,P}, and so on.

Extriangulated categories were recently introduced by Nakaoka and Palu \cite{NP} by extracting those properties of ${\rm \Ext}^1$ on exact categories and on triangulated categories that seem relevant from the point of view of cotorsion pairs. In particular, triangulated categories and exact categories are extriangulated categories. There are a lot of examples of extriangulated categories which are neither triangulated categories nor exact categories, see \cite{NP,ZZ,HZZ,ZhZ,NP1}.

 Under the assumption of the (WIC) condition, Wang-Wei-Zhang \cite{WWZ} introduced the notion of compatible morphisms. Using this notion, they modified the extriangulated functor between two extriangulated categories, which has been defined in \cite{B-TS}, and then, introduced the concepts of right (left) exact functors in extriangulated categories. By these frameworks,
Wang-Wei-Zhang \cite{WWZ} gave a simultaneous generalization of recollements of abelian categories
and triangulated categories, which we call  the recollement of extriangulated categories. As an application, they glued the cotorsion pairs along recollements of extriangulated categories, which covers the results obtained by Chen in \cite{ChenJM}. Later, He-Hu-Zhou \cite{HHZ}  glued the torsion pairs via  the recollements of extriangulated categories. This unifies their results of Chen \cite{ChenJM}
and Ma-Huang \cite{MH} in the framework of extriangulated categories.
 Recently, some homological aspects of recollements of extriangulated categories were investigated. For instance, Gu-Ma-Tan \cite{GMT} studied the  behavior of homological dimension along the recollements of extriangulated categories.

In this paper, we devote to giving a generalization of the recollements of extriangulated categories in the sense of \cite{WWZ}. Without  the (WIC) condition, we introduce a new concept, called $\s$-recollements of extriangulated categories. Moreover,  in this definition, we do not need to modify the extriangulated functor (in this paper, we call exact functor) in the sense of \cite{B-TS}. Because of this point, the notion of $\s$-recollements looks more natural. We present the basic properties of $\s$-recollements, which are also appeared in \cite{WWZ}.
We also  show that  recollements of abelian and triangulated categories, as well as recollements of extriangulated categories, are all $\s$-recollements. To clarify  the relationship among them, we provide the following figure.
$$
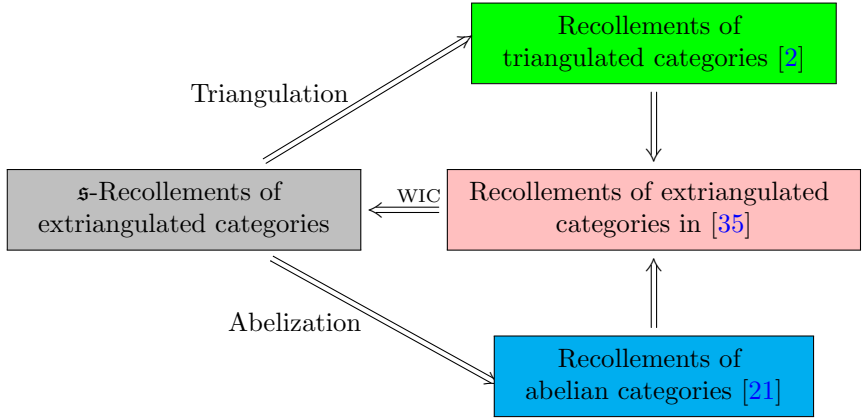
\xymatrix@R=1em@C=1em{&&\fcolorbox{black}{green}{\text{$\begin{array}{c}
                           \textrm{Recollements~of} \\
                           \textrm{ triangulated~categories \cite{BBD}}
                          \end{array}$}
}\ar@{=>}[2,0]&&&\\
&&&&&\\
\fcolorbox{black}{lightgray}{\text{$\begin{array}{c}
                           \s\textrm{-Recollements~of} \\
                           \textrm{ extriangulated~categories}
                          \end{array}$}
}\ar@{=>}[-2,2]-<24mm,-1mm>^*+\txt{Triangulation}\ar@{=>}[2,2]+<-20.5mm,-1mm>_*+{\textrm{Abelization}}&&\ar@{=>}[ll]_{\textrm{WIC}}\fcolorbox{black}{pink}{\text{$\begin{array}{c}
                           \textrm{Recollements~of~extriangulated~} \\
                           \textrm{categories~in \cite{WWZ}}
                          \end{array}$}}&&&\\
                          &&&&&\\
                          &&\fcolorbox{black}{cyan}{\text{$\begin{array}{c}
                           \textrm{Recollements~of} \\
                           \textrm{ abelian~categories \cite{MV}}
                          \end{array}$}
}\ar@{=>}[-2,0]&&&}$$
On the other hand, we study the localization theory on the adjoint pairs, where the localization theory of extriangulated categories was introduced by \cite{NP1} and unified the Serre quotient of abelian categories and the Verdier quotient of triangulated categories. Under some mild conditions, we prove that each adjoint pair of exact functors between two extriangulated categories induces a new adjoint pair of exact functors between  the localizations of extriangulated categories, see Theorem \ref{prop2-1}.

Recollements are closely related to localization. It is well-known that there is a  bijection  between recollements of triangulated categories and torsion torsion-free triples (TTF-triples, for short), see \cite{BBD,BR,N,N01}.  Psaroudakis and Vit\'{o}ria \cite[Theorem 4.3]{P01} shown that recollements
of abelian categories (up to equivalence) are in bijection with its bilocalising TTF classes. Because of these points, one can regard $\mathcal{C}$ in the diagram (\ref{dia0-0}) of triangulated  or abelian categories  as a localization of $\mathcal{B}$. In fact, it is a Verdier quotient or Serre quotient of $\mathcal{B}$.  Recently, Nakaoka,  Ogawa, and Sakai prove that in the context of recollements of extriangulated categories, $\mathcal{C}$ is also a localization of $\mathcal{B}$ in the sense of \cite{NP1}. In this note, we prove that this phenomenon still works in the situation of $\s$-recollements, see Proposition \ref{prop3-5}.

The paper is organized as follows: we collect some basic definitions and properties
of extriangulated categories and $\s$-exact functors in Section 2. In Section 3, we discuss the localization theory of adjoint pairs. In Section 4,
we introduce the $\s$-recollement of extriangulated categories and give some basic properties. In Section 5, applying the localization theory into $\s$-recollements, we construct a recollement of triangulated categories from an $\s$-recollement.

\section{Preliminaries}
\smallskip

\subsection{Some basic notions and properties in extriangulated categories}
~\par
Let us briefly recall some definitions and basic properties of extriangulated categories from \cite{NP}.
We omit some details here, but the reader can find
them in \cite{NP}.

Let $\mathcal{C}$ be an additive category equipped with an additive bifunctor
$$\mathbb{E}: \mathcal{C}^{\rm op}\times \mathcal{C}\rightarrow {\rm Ab},$$
where ${\rm Ab}$ is the category of abelian groups. For any objects $A, C\in\mathcal{C}$, an element $\delta\in \mathbb{E}(C,A)$ is called an $\mathbb{E}$-extension.
Let $\mathfrak{s}$ be a correspondence which associates an equivalence class $$\mathfrak{s}(\delta)=\xymatrix@C=0.8cm{[A\ar[r]^x
 &B\ar[r]^y&C]}$$ to any $\mathbb{E}$-extension $\delta\in\mathbb{E}(C, A)$. This $\mathfrak{s}$ is called a {\it realization} of $\mathbb{E}$, if it makes the diagrams in \cite[Definition 2.9]{NP} commutative.
 A triplet $(\mathcal{C}, \mathbb{E}, \mathfrak{s})$ is called an {\it extriangulated category} if it satisfies the following conditions.
\begin{enumerate}
\item[\rm (1)] $\mathbb{E}\colon\mathcal{C}^{\rm op}\times \mathcal{C}\rightarrow \rm{Ab}$ is an additive bifunctor.
\vspace{0.5mm}

\item[\rm (2)] $\mathfrak{s}$ is an additive realization of $\mathbb{E}$.
\vspace{0.5mm}

\item[\rm (3)] $\mathbb{E}$ and $\mathfrak{s}$  satisfy the compatibility
conditions in \cite[Definition 2.12]{NP}.
 \end{enumerate}

\begin{remark}
Note that both exact categories and triangulated categories are extriangulated categories, see \cite[Example 2.13]{NP} and extension closed subcategories of triangulated categories are
again extriangulated, see \cite[Remark 2.18]{NP}. Moreover, there exist extriangulated categories which
are neither exact categories nor triangulated categories, see \cite{NP,ZZ,HZZ,ZhZ,NP1}.
\end{remark}

We will use the following terminology, which comes from \cite{NP}.
\begin{definition}
Let $(\C,\E,\s)$ be an extriangulated category.
\begin{itemize}
\item[(1)] A sequence $A\xrightarrow{~x~}B\xrightarrow{~y~}C$ is called a {\it conflation} if it realizes some $\E$-extension $\del\in\E(C,A)$.
    In this case, $x$ is called an {\it inflation} and $y$ is called a {\it deflation}.

\item[(2)] If a conflation  $A\xrightarrow{~x~}B\xrightarrow{~y~}C$ realizes $\delta\in\mathbb{E}(C,A)$, we call the pair $( A\xrightarrow{~x~}B\xrightarrow{~y~}C,\delta)$ an {\it $\E$-triangle}, and write it in the following way.
$$A\overset{x}{\longrightarrow}B\overset{y}{\longrightarrow}C\overset{\delta}{\dashrightarrow}$$
We usually do not write this $``\delta"$ if it is not used in the argument.

\item[(3)] Let $A\overset{x}{\longrightarrow}B\overset{y}{\longrightarrow}C\overset{\delta}{\dashrightarrow}$ and $A^{\prime}\overset{x^{\prime}}{\longrightarrow}B^{\prime}\overset{y^{\prime}}{\longrightarrow}C^{\prime}\overset{\delta^{\prime}}{\dashrightarrow}$ be any pair of $\E$-triangles. If a triplet $(a,b,c)$ realizes $(a,c)\colon\delta\to\delta^{\prime}$, then we write it as
$$\xymatrix{
A \ar[r]^x \ar[d]^a & B\ar[r]^y \ar[d]^{b} & C\ar@{-->}[r]^{\del}\ar[d]^c&\\
A'\ar[r]^{x'} & B' \ar[r]^{y'} & C'\ar@{-->}[r]^{\del'} &}$$
and call $(a,b,c)$ a {\it morphism of $\E$-triangles}.

\item[(4)] An object $P\in\C$ is called {\it projective} if
for any $\E$-triangle $A\overset{x}{\longrightarrow}B\overset{y}{\longrightarrow}C\overset{\delta}{\dashrightarrow}$ and any morphism $c\in\C(P,C)$, there exists $b\in\C(P,B)$ satisfying $yb=c$.
We denote the subcategory of projective objects by $ \mathcal{P}\subseteq\C$. Dually, the subcategory of injective objects is denoted by $ \mathcal{I}\subseteq\C$.

\item[(5)] We say that $\C$ {\it has enough projective objects} if
for any object $C\in\C$, there exists an $\E$-triangle
$A\overset{x}{\longrightarrow}P\overset{y}{\longrightarrow}C\overset{\delta}{\dashrightarrow}$
satisfying $P\in \mathcal{P}$. We can define the notion of having enough injectives dually.

\end{itemize}
\end{definition}

For an conflation $A\xrightarrow{x} B
\xrightarrow{y} C$ in an extriangulated category, we write
$C = \textrm{Cone}(x)$ and call it a cone of $x$. This is uniquely determined by $x$ up to
isomorphisms. Dually we write $A = \textrm{CoCone}(y)$ and call it a cocone of $y$, which is
uniquely determined by y up to isomorphisms.\par

\begin{remark}\cite[Proposition 3.24]{NP}
Let $(\C, \E, \s)$ be an extriangulated category.
Then

(a) $P$ is projective object in $\C$ if and only if $\E(P,C)$=0, for any $C\in\C$.

(b) $I$ is injective object in $\C$ if and only if $\E(C,I)$=0, for any $C\in\C$.
\end{remark}
We modify the notions of right (left) exact sequences and right (left) exact functors in \cite{WWZ}.
\begin{definition}Let $(\C, \E, \s)$ be an extriangulated category.
 A sequence $A\stackrel{f}{\longrightarrow}B\stackrel{g}{\longrightarrow}C$ in $\mathcal{C}$ is said to be {\em right $\s$-exact} if
there exists an $\mathbb{E}$-triangle $K\stackrel{h_{2}}{\longrightarrow}B\stackrel{g}{\longrightarrow}C\stackrel{}\dashrightarrow$ and a deflation $h_{1}:A\rightarrow K$, such that $f=h_2h_1$. Dually one can also define the {\em  left $\s$-exact} sequences.


A $4$-term $\mathbb{E}$-triangle sequence $A{\stackrel{f}\longrightarrow}B\stackrel{g}{\longrightarrow}C\stackrel{h}{\longrightarrow}D$ is called {\em  $\s$-exact}  if there exist $\mathbb{E}$-triangles $A\stackrel{f}{\longrightarrow}B\stackrel{g_1}{\longrightarrow}K\stackrel{}\dashrightarrow$
and $K\stackrel{g_{2}}{\longrightarrow}C\stackrel{h}{\longrightarrow}D\stackrel{}\dashrightarrow$ such that $g=g_2g_1$.
\end{definition}

\begin{definition}
Let $(\mathcal{A},\mathbb{E}_{\mathcal{A}},\mathfrak{s}_{\mathcal{A}})$ and $(\mathcal{B},\mathbb{E}_{\mathcal{B}},\mathfrak{s}_{\mathcal{B}})$ be extriangulated categories. An additive covariant functor $F:\mathcal{A}\rightarrow \mathcal{B}$ is called a {\em right $\s$-exact functor} if it satisfies the following conditions
\begin{itemize}
   \item [(1)] If $A\stackrel{a}{\longrightarrow}B\stackrel{b}{\longrightarrow}C$ is right exact in $\mathcal{A}$, then $FA\stackrel{Fa}{\longrightarrow}FB\stackrel{Fb}{\longrightarrow}FC$ is right exact in $\mathcal{B}$. (Then for any $\mathbb{E}_{\mathcal{A}}$-triangle $A\stackrel{f}{\longrightarrow}B\stackrel{g}{\longrightarrow}C\stackrel{\delta}\dashrightarrow$, there exists an $\mathbb{E}_{\mathcal{B}}$-triangle $A'\stackrel{x}{\longrightarrow}FB\stackrel{Fg}{\longrightarrow}FC\stackrel{}\dashrightarrow$ such that $Ff=xy$ and $y: FA\rightarrow A'$ is a deflation.  Moreover, $A'$ is uniquely determined up to isomorphism.)
 \item [(2)] There exists a natural transformation $$\eta=\{\eta_{(C,A)}:\mathbb{E}_{\mathcal{A}}(C,A)\longrightarrow\mathbb{E}_{\mathcal{B}}(F^{op}C,A')\}_{(C,A)\in{\mathcal{A}}^{\rm op}\times\mathcal{A}}$$ such that $\mathfrak{s}_{\mathcal{B}}(\eta_{(C,A)}(\delta))=[A'\stackrel{x}{\longrightarrow}FB\stackrel{Fg}{\longrightarrow}FC]$.

 \end{itemize}
Dually, we define the {\em left  $\s$-exact functor} between two extriangulated categories.
\end{definition}

\begin{definition}[\cite{B-TS}] 
Let $(\mathcal{A},\mathbb{E}_{\mathcal{A}},\mathfrak{s}_{\mathcal{A}})$ and $(\mathcal{B},\mathbb{E}_{\mathcal{B}},\mathfrak{s}_{\mathcal{B}})$ be extriangulated categories.  We say an additive covariant functor $F:\mathcal{A}\rightarrow \mathcal{B}$ is an {\em $\s$-exact functor} (or, simply exact functor) if there exists a natural transformation $$\eta=\{\eta_{(C,A)}\}_{(C,A)\in{\mathcal{A}}^{\rm op}\times\mathcal{A}}:\mathbb{E}_{\mathcal{A}}(-,-)\Rightarrow\mathbb{E}_{\mathcal{B}}(F(-),F(-)).$$
which satisfies $\mathfrak{s}_{\mathcal{B}}(\eta_{(C,A)}(\delta))=[F(A)\stackrel{F(x)}{\longrightarrow}F(B)\stackrel{F(y)}{\longrightarrow}F(C)]$.
for any $\mathfrak{s}_{\mathcal{A}}(\delta)=[A\stackrel{x}{\longrightarrow}B\stackrel{y}{\longrightarrow}C]$.
\end{definition}
It is easy to obtain the following result.
\begin{proposition}
 Let $(\mathcal{A},\mathbb{E}_{\mathcal{A}},\mathfrak{s}_{\mathcal{A}})$ and $(\mathcal{B},\mathbb{E}_{\mathcal{B}},\mathfrak{s}_{\mathcal{B}})$ be extriangulated categories. An additive covariant functor $F: \mathcal{A}\rightarrow\mathcal{B}$ is exact if and only if $F$ is both left $\s$-exact and $\s$-right exact.
\end{proposition}

We recall some properties about adjoint pairs, which can be found in \cite{HS} and
\cite{W}.

\begin{lemma}\label{lem2-1}
Let {\rm($F$, $G$)} be an adjoint pair of additive categories with $F: \A\rightarrow \B$ and $G: \B\rightarrow \A$. Denote by $\eta: id_{\mathcal{A}}\rightarrow G F$ the unit map and by $\varepsilon:F G \rightarrow id_{\mathcal{B}}$ the counit map.
\begin{enumerate}
  \item $F$ is fully faithful if and only if the unit map $\eta$ is a natural isomorphism.
  \item $G$ is fully faithful if and only if the counit map $\varepsilon$ is a natural isomorphism.
  \item Both the following composites are the identities
  \begin{align*}
   F(X)\xrightarrow{F(\eta_{_{X}})}FGF(X)\xrightarrow{\varepsilon_{_{F(X)}}}F(X)\\
   G(Y)\xrightarrow{\eta_{G(Y)}}GFG(Y)\xrightarrow{G(\varepsilon_{_{Y}})}G(Y).
  \end{align*}
  \item If $F$  is fully faithful, then $F\eta_{X}=\varepsilon^{-1}_{FX}$ and $G\varepsilon_{Y}=\eta^{-1}_{GY}$, for any $X\in\A$ and $Y\in\B$.
  \item If $G$  is fully faithful, then $F\eta_{X}=\varepsilon^{-1}_{FX}$ and $G\varepsilon_{Y}=\eta^{-1}_{GY}$, for any $X\in\A$ and $Y\in\B$.
\end{enumerate}
\end{lemma}

\section{Localization theory of adjoint pairs}\label{sect2.3}
\vspace{1mm}

Let $\A$ be an additive category. We denote by $\mathcal{M}_{\A}$ the class of all morphisms in $\A$. If a class of morphisms $\mathcal{I}\subseteq \mathcal{M}_{\A}$
is closed by compositions and contains all identities in $\A$, then we may regard
$\mathcal{I}\subseteq\A$ as a (not full) subcategory satisfying $Ob(\I) = Ob(\A)$. With this view in
mind, we write $\I(X,Y) = \{f \in \A(X,Y) | f \in \I\}$ for any $X$, $Y\in\A$.  We denote by $\textrm{Iso}\A$ the class of all isomorphisms in $\A$. $N_{\mathcal{I}}$ denotes the full subcategory consisting of objects $N\in \A$ such that both $N\rightarrow 0$ and $0\rightarrow N$ belong to $\mathcal{I}$. It is obvious that $N_{\mathcal{I}}\subseteq \A$ is an additive subcategory. $\overline{\A}=\A/[N_{\mathcal{I}}]$ denotes the ideal quotient of $\A$. $\overline{\I}$ denotes a set of morphisms in the ideal quotient
$\A/[N_{\mathcal{I}}]$ obtained from $\I$ by taking closure with respect to the composition with isomorphisms in $\A$. \par
Now, let $(\mathcal{A},\mathbb{E},\mathfrak{s})$  be an extriangulated category and  $\mathcal{I}\subseteq \mathcal{M}_{\A}$  a set of morphisms which satisfies the following conditions.
\begin{enumerate}
  \item[(M0)] $\I$ contains all isomorphisms in $\A$, and closed by compositions. Also, $\I$ is closed by taking finite direct sums. Namely, if $f_{i}\in\I(X_{i},Y_{i})$, for $i=1,2$, then $f_{1}\oplus f_{2}\in\I(X_{1}\oplus X_{2},Y_{1}\oplus Y_{2})$.
  \item[(MR1)] $\overline{\I}$ satisfies 2-out-of-3 with respect to compositions.
  \item[(MR2)] $\overline{\I}$ is a multiplicative system.
  \item[(MR3)] Let $A\xrightarrow{x} B \xrightarrow{y} C\overset{\delta}{\dashrightarrow}$ and $A'\xrightarrow{x'} B' \xrightarrow{y'} C'\overset{\delta'}{\dashrightarrow}$ be any pair of $\E$-triangles, and let $a\in \I(A,A')$, $c\in \I(C,C')$ be any pair of morphisms
satisfying $a_{\ast}\delta = c^{\ast}\delta'$. If $a$ and $c$ belong to $\I$, then there exists $\mathbf{b} \in \I(B,B')$ which satisfies $\mathbf{b} \circ \overline{x} = \overline{x'}\circ a$ and $c \circ \overline{y} = \overline{y'}\circ \mathbf{b}$.
  \item[(MR4)] $\{\mathbf{v} \circ \overline{x} \circ \mathbf{u} ~| ~x ~\textrm{is ~an ~inflation}, ~\mathbf{u}, ~\mathbf{v} \in \I\}\subseteq \overline{\M_{\A}}$ is closed by compositions. $\{\mathbf{v} \circ \overline{y} \circ \mathbf{u}~ |~ y~ \textrm{is~an~deflation}, ~\mathbf{u}, ~\mathbf{v} ~\in \I\}\subseteq \overline{\M_{\A}}$ is closed
by compositions.
\end{enumerate}
\begin{theorem}{\rm\cite{NOS}}\label{th2-1}
Let $(\mathcal{A},\mathbb{E},\mathfrak{s})$  be an extriangulated category and  $\mathcal{I}\subseteq \mathcal{M}_{\A}$  a set of morphisms which satisfies {\rm(M0)}, {\rm(MR1)}, {\rm(MR2)}, {\rm(MR3)} and {\rm(MR4)}. Then the localization of $\A$ by $\I$ gives an extriangulated category $(\widetilde{\mathcal{A}},\widetilde{\mathbb{E}},\widetilde{\mathfrak{s}})$ equipped
with an exact functor $Q:\mathcal{A}\rightarrow \widetilde{\mathcal{A}}$ with a  natural transformation $\mu:\E\Rightarrow \widetilde{\E}\circ(Q^{op}\times Q)$, which is universal among exact functors inverting $\I$.
\end{theorem}
\begin{remark}  In details, the exact functor $Q:\mathcal{A}\rightarrow \widetilde{\mathcal{A}}$ is a composition of the quotient functor $p:\mathcal{A}\rightarrow \overline{\A}$ and the the localization $\overline{Q}:\overline{\A}\rightarrow \widetilde{\mathcal{A}}$ by the  multiplicative system $\overline{\I}$. In this view,
 the objects of $\widetilde{\mathcal{A}}$ is same as $\overline{\A}$ and the morphism set $\Hom_{\widetilde{\mathcal{A}}}(\overline{X}, \overline{Y})$  is a set of  equivalence classes of right roofs which are the diagrams of the form
 $$\xymatrix{&Z\ar@{=>}[dl]_{\overline{s}}\ar[dr]^{\overline{f}}&\\
 \overline{X}&&\overline{Y}    }$$
  or simply written as the right fractions  $\overline{f}/\overline{s}$, where $\overline{s}\in\overline{\I}$ and $\overline{f}\in \Hom_{\overline{\A}}(\overline{Z},\overline{Y})$.
\end{remark}

\begin{corollary}\label{cor2-1}
Let $(\mathcal{A},\mathbb{E}_{\A},\mathfrak{s}_{\A})$  and  $(\mathcal{B},\mathbb{E}_{\B},\mathfrak{s}_{\B})$ be extriangulated categories and  $\mathcal{I}_{\A}\subseteq \mathcal{M}_{\A}$  and $\mathcal{I}_{\B}\subseteq \mathcal{M}_{\B}$ two sets of morphisms which satisfy {\rm(M0)}, {\rm(MR1)}, {\rm(MR2)}, {\rm(MR3)} and {\rm(MR4)}. If the exact functor $F:\mathcal{A}\rightarrow \B$ sends $\mathcal{I}_{\A}$ to $\mathcal{I}_{\B}$, then there exists an exact functor $\widetilde{F}:\widetilde{\mathcal{A}}\rightarrow \widetilde{\mathcal{B}}$ and the following commutative diagram
$$\xymatrix{ \mathcal{A}\ar[r]^{Q_{\A}}\ar[d]_-{F}&\widetilde{\mathcal{A}}\ar[d]^-{\widetilde{F}}\\
\mathcal{B}\ar[r]^{Q_{\B}}&\widetilde{\mathcal{B}} }$$
\end{corollary}
\begin{proof} For any $f\in\mathcal{I}_{\A}$, $Q_{\B}F(f)$ is an isomorphism in $\widetilde{\mathcal{B}}$. Then, by Theorem \ref{th2-1}, there exists an exact functor $\widetilde{F}:\widetilde{\mathcal{A}}\rightarrow \widetilde{\mathcal{B}}$ such that $\widetilde{F}Q_{\A}=Q_{\B}F$.
\end{proof}

\begin{theorem}\label{prop2-1}
Let $(\mathcal{A},\mathbb{E}_{\A},\mathfrak{s}_{\A})$  and  $(\mathcal{B},\mathbb{E}_{\B},\mathfrak{s}_{\B})$ be extriangulated categories and  $\mathcal{I}_{\A}\subseteq \mathcal{M}_{\A}$  and $\mathcal{I}_{\B}\subseteq \mathcal{M}_{\B}$ two sets of morphisms which satisfy $\overline{\mathcal{I}_{\A}}=p(\mathcal{I}_{\A})$, $\overline{\mathcal{I}_{\B}}=p(\mathcal{I}_{\B})$, {\rm(M0)}, {\rm(MR1)}, {\rm(MR2)}, {\rm(MR3)} and {\rm(MR4)}.   If  $F:\mathcal{A}\rightarrow \B$ sends $\mathcal{I}_{\A}$ to $\mathcal{I}_{\B}$ and $F$ admits a  right adjoint  functor $G$ such that $G$ is exact and $G:\mathcal{B}\rightarrow \A$ sends $\mathcal{I}_{\B}$ to $\mathcal{I}_{\A}$, then there exists an adjoint pair {\rm($\widetilde{F}$, $\widetilde{G}$)}, where both $\widetilde{F}$ and $\widetilde{G}$ are exact functors.
\end{theorem}
\begin{proof}
From Corollary \ref{cor2-1}, we have two exact functor $F:\widetilde{\mathcal{A}}\rightarrow \widetilde{\mathcal{B}}$ and $F:\widetilde{\mathcal{B}}\rightarrow \widetilde{\mathcal{A}}$ such that the following diagrams commutate
\begin{align*}
  \xymatrix{ \mathcal{A}\ar[r]^{Q_{\A}}\ar[d]_-{F}&\widetilde{\mathcal{A}}\ar[d]^-{\widetilde{F}}\\
\mathcal{B}\ar[r]^{Q_{\B}}&\widetilde{\mathcal{B}} } &   & \xymatrix{\mathcal{B}\ar[r]^{Q_{\B}}\ar[d]_-{G}&\widetilde{\mathcal{B}}\ar[d]^-{\widetilde{G}}\\
\mathcal{A}\ar[r]^{Q_{\A}}&\widetilde{\mathcal{A}} }
\end{align*}
For any $X\in N_{\mathcal{I}_{\A}}$, both $X\rightarrow 0$ and $0\rightarrow X$ are belong to $\mathcal{I}_{\A}$. Since $F(\mathcal{I}_{\A})\subseteq \mathcal{I}_{\B}$, both $F(X)\rightarrow 0$ and $0\rightarrow F(X)$ are belong to $\mathcal{I}_{\B}$. Then $F(N_{\mathcal{I}_{\A}})\subseteq N_{\mathcal{I}_{\B}}$. Hence, $F$ induces the quotient group homomorphism
$$\frac{\Hom_{\A}(X,Y)}{[N_{\mathcal{I}_{\A}}](X,Y)}\longrightarrow \frac{\Hom_{\B}(F(X),F(Y))}{[N_{\mathcal{I}_{\B}}](F(X),F(Y))}.$$
Similarly, $G(N_{\mathcal{I}_{\B}})\subseteq N_{\mathcal{I}_{\A}}$ and $G$ induces the quotient group homomorphism
$$\frac{\Hom_{\B}(X,Y)}{[N_{\mathcal{I}_{\B}}](X,Y)}\longrightarrow \frac{\Hom_{\A}(G(X),G(Y))}{[N_{\mathcal{I}_{\A}}](G(X),G(Y))}.$$
Next, we prove that ($\widetilde{F}$, $\widetilde{G}$) is an adjoint pair. We divided the proof into several steps.\par
\textbf{Step 1.}
Now, we construct the following map
\begin{align*}
  \Phi_{\overline{X}, \overline{Y}}: &\Hom_{\widetilde{\mathcal{B}}}(\widetilde{F}(\overline{X}), \overline{Y}) \longrightarrow \Hom_{\widetilde{\mathcal{A}}}(\overline{X}, \widetilde{G}(\overline{Y}))
\end{align*}
Note that
\begin{align*}
   \Hom_{\widetilde{\mathcal{B}}}(\widetilde{F}(\overline{X}), \overline{Y})&= \Hom_{\widetilde{\mathcal{B}}}(\widetilde{F}(Q_{\A}(X)), \overline{Y}) \\
   &=\Hom_{\widetilde{\mathcal{B}}}(Q_{\B}(F(X)), \overline{Y})\\
   &=\Hom_{\widetilde{\mathcal{B}}}(\overline{F(X)}, \overline{Y})
\end{align*}
Similarly, we have $\Hom_{\widetilde{\mathcal{A}}}(\overline{X}, \widetilde{G}(\overline{Y}))= \Hom_{\widetilde{\mathcal{A}}}(\overline{X}, \overline{G(Y)})$.

Let $\overline{f}/\overline{s}\in\Hom_{\widetilde{\mathcal{B}}}(\overline{F(X)}, \overline{Y})$, where $\overline{f}\in \overline{\A}(\overline{\bullet},  \overline{Y})$ and $\overline{s}\in\overline{\I_{\A}}(\overline{\bullet},\overline{F(X)})$. Since $\overline{\I_{\mathcal{B}}}$  is a multiplicative system and $\overline{G(\I_{\mathcal{B}})}\subseteq \overline{\I_{\mathcal{A}}}$, taking the homotopy pullback of $\overline{\eta_{X}}$ and $\overline{G(s)}$, we have the following commutative diagram
$$ \xymatrix{ &\overline{G(Y)}\\ \overline{A}\ar[ur]^{\overline{f'}}\ar@{=>}[d]_{\overline{y}}\ar@{}[dr]|-{\textrm{PB}}\ar[r]^{\overline{x}}&\overline{G(\bullet)}\ar@{=>}[d]^{\overline{G(s)}}\ar[u]_{\overline{G(f)}}\\
\overline{X}\ar[r]^-{\overline{\eta_{X}}}&\overline{GF(X)}}$$
where $\eta_{X}$ is an unit map, $y\in \I_{\A}(A,X)$ and $\overline{f'}=\overline{G(f)x}$. Then we define $\Phi_{\overline{X}, \overline{Y}}(\overline{f}/\overline{s})=\overline{G(f)x}/\overline{y}$. In order to show that the definition of $\Phi_{\overline{X}, \overline{Y}}$ is well-defined, let $\overline{f_{1}}/\overline{s_{1}}=\overline{f_{2}}/\overline{s_{2}}$. Then there is a commutative diagram
$$\xymatrix{   &\overline{Z_{1}}\ar@{=>}[dl]_{\overline{s_{1}}}\ar[dr]-<-0.5mm,-2.7mm>^{\overline{f_{1}}}&\\
\overline{F(X)}&\ar@{=>}[l]_-{\overline{u}}\overline{\bullet}\ar[u]^-{\overline{t_{1}}}\ar[d]_-{\overline{t_{2}}}\ar[r]&\overline{Y}\\
&\overline{Z_{2}}\ar@{=>}[ul]^{\overline{s_{2}}}\ar[ur]-<1mm,2.2mm>_{\overline{f_{2}}}&        }$$
From this diagram, we take a reduction of fractions $\overline{f_{1}}/\overline{s_{1}}$ and $\overline{f_{2}}/\overline{s_{2}}$ to a common denominator. Then we have
\begin{align*}
\overline{f_{1}}/\overline{s_{1}}-\overline{f_{2}}/\overline{s_{2}}=\overline{f_{1}t_{1}}/\overline{u}-\overline{f_{2}t_{2}}/\overline{u}=(\overline{f_{1}t_{1}}-\overline{f_{2}t_{2}})/\overline{u}=(\overline{f_{1}t_{1}-f_{2}t_{2}})/\overline{u}=0/\overline{u}=0
\end{align*}
Hence,  exists an object $N\in N_{\mathcal{I}_{\B}}$ such that $f_{1}t_{1}-f_{2}t_{2}$ factors through $N$ that is, there is a commutative diagram.
$$\xymatrix{ \bullet\ar[rr]^{f_{1}t_{1}-f_{2}t_{2}} \ar[dr]& &Y\\
&N\ar[ur]&            }$$
From the right fractions $\overline{f_{1}t_{1}}/\overline{u}$ and $\overline{f_{2}t_{2}}/\overline{u}$, we have the following two diagrams
\begin{align*}
  \xymatrix{ &\overline{G(Y)}\\ \overline{A}\ar[ur]^{\overline{f_{1}'}}\ar@{=>}[d]_{\overline{y}}\ar@{}[dr]|-{\textrm{PB}}\ar[r]^{\overline{x}}&\overline{G(\bullet)}\ar@{=>}[d]^{\overline{G(u)}}\ar[u]_{\overline{G(f_{1}t_{1})}}\\
\overline{X}\ar[r]^-{\overline{\eta_{X}}}&\overline{GF(X)}}& & \xymatrix{ &\overline{G(Y)}\\ \overline{A}\ar[ur]^{\overline{f'_{2}}}\ar@{=>}[d]_{\overline{y}}\ar@{}[dr]|-{\textrm{PB}}\ar[r]^{\overline{x}}&\overline{G(\bullet)}\ar@{=>}[d]^{\overline{G(u)}}\ar[u]_{\overline{G(f_{2}t_{2})}}\\
\overline{X}\ar[r]^-{\overline{\eta_{X}}}&\overline{GF(X)}}
\end{align*}
 Then we have
 \begin{align*}
   \Phi_{\overline{X}, \overline{Y}}(\overline{f_{1}}/\overline{s_{1}})-\Phi_{\overline{X}, \overline{Y}}(\overline{f_{2}}/\overline{s_{2}}) & =\Phi_{\overline{X}, \overline{Y}}(\overline{f_{1}t_{1}}/\overline{u})-\Phi_{\overline{X}, \overline{Y}}(\overline{f_{2}t_{2}}/\overline{u}) \\
   &=\overline{G(f_{1}t_{1})x}/\overline{y}-\overline{G(f_{2}t_{2})x}/\overline{y}\\
   &=(\overline{G(f_{1}t_{1})x}-\overline{G(f_{2}t_{2})x})/\overline{y}\\
   &=(\overline{G(f_{1}t_{1}-f_{2}t_{2})}\circ \overline{x})/\overline{y}\\
   &=(0\circ \overline{x})/\overline{y}\\
   &=0
 \end{align*}
 where $\overline{G(f_{1}t_{1}-f_{2}t_{2})}=0$ is because   $G(f_{1}t_{1}-f_{2}t_{2})$  factors through $G(N)$ and $G(N)\in N_{\mathcal{I}_{\A}}$. Thus, $\Phi_{\overline{X}, \overline{Y}}$ is well-defined.\par
\textbf{ Step 2.} Now, we construct the following map
\begin{align*}
  \Psi_{\overline{X}, \overline{Y}}: & \Hom_{\widetilde{\mathcal{A}}}(\overline{X}, \widetilde{G}(\overline{Y})) \longrightarrow\Hom_{\widetilde{\mathcal{B}}}(\widetilde{F}(\overline{X}), \overline{Y}).
\end{align*}\par
 Let $\overline{g}/\overline{t}\in\Hom_{\widetilde{\mathcal{A}}}(\overline{X}, \overline{G(Y)}) $, where $\overline{g}\in \overline{\A}(\overline{\maltese},  \overline{G(Y)})$ and $\overline{t}\in\overline{\I_{\A}}(\overline{\maltese},\overline{X})$. Since $F(\I_{\A})\subseteq \I_{\B}$, we define that $\Psi_{\overline{X}, \overline{Y}}(\overline{g}/\overline{t})=\overline{\varepsilon_{Y}F(g)}/\overline{F(t)}$, where $\varepsilon:FG\rightarrow id$ is the counit map.
In order to show that the definition of $\Psi_{\overline{X}, \overline{Y}}$ is well-defined, let $\overline{g_{1}}/\overline{t_{1}}=\overline{g_{2}}/\overline{t_{2}}$. Then there is a commutative diagram
$$\xymatrix{   &\overline{Z_{1}}\ar@{=>}[dl]_{\overline{t_{1}}}\ar[dr]^{\overline{g_{1}}}&\\
\overline{X}&\ar@{=>}[l]_-{\overline{u}}\overline{\maltese}\ar[u]^-{\overline{s_{1}}}\ar[d]_-{\overline{s_{2}}}\ar[r]&\overline{G(Y)}\\
&\overline{Z_{2}}\ar@{=>}[ul]^{\overline{t_{2}}}\ar[ur]_{\overline{g_{2}}}&        }$$
From this diagram, we take a reduction of fractions $\overline{g_{1}}/\overline{t_{1}}$ and $\overline{g_{2}}/\overline{t_{2}}$ to a common denominator. Then we have
\begin{align*}
\overline{g_{1}}/\overline{t_{1}}-\overline{g_{2}}/\overline{t_{2}}=\overline{g_{1}s_{1}}/\overline{u}-\overline{g_{2}s_{2}}/\overline{u}=(\overline{g_{1}s_{1}}-\overline{g_{2}s_{2}})/\overline{u}=(\overline{g_{1}s_{1}-g_{2}s_{2}})/\overline{u}=0/\overline{u}=0
\end{align*}
Hence,  exists an object $N'\in N_{\mathcal{I}_{\A}}$ such that $g_{1}s_{1}-g_{2}s_{2}$ factors through $N'$ that is, there is a commutative diagram.
$$\xymatrix{ \maltese\ar[rr]^{g_{1}s_{1}-g_{2}s_{2}} \ar[dr]& &Y\\
&N'\ar[ur]&            }$$
Note that
\begin{align*}
  \Psi_{\overline{X}, \overline{Y}}(\overline{g_{1}}/\overline{t_{1}})-\Psi_{\overline{X}, \overline{Y}}(\overline{g_{2}}/\overline{t_{2}}) =& =\Psi_{\overline{X},\overline{Y}}(\overline{g_{1}s_{1}}/\overline{u})-\Psi_{\overline{X}, \overline{Y}}(\overline{g_{2}s_{2}}/\overline{u}) \\
   & =\overline{\varepsilon_{Y}F(g_{1}s_{1})}/\overline{F(u)}-\overline{\varepsilon_{Y}F(g_{2}s_{2})}/\overline{F(u)}\\
&=(\overline{\varepsilon_{Y}}\circ\overline{F(g_{1}s_{1}-g_{2}s_{2})})/\overline{F(u)}\\
&=(\overline{\varepsilon_{Y}}\circ0)/\overline{F(u)}\\
&=0
\end{align*}
Here, $\overline{F(g_{1}s_{1}-g_{2}s_{2})}=0$ is because $F(g_{1}s_{1}-g_{2}s_{2})$ factors through $F(N')$ and $F(N_{\mathcal{I}_{\A}})\subseteq N_{\mathcal{I}_{\B}}$.
Thus, $\Psi_{\overline{X}, \overline{Y}}$ is well-defined.\par
\textbf{Step 3.} Next, we shall check that $\Psi_{\overline{X}, \overline{Y}}\Phi_{\overline{X}, \overline{Y}}(\overline{f}/\overline{s})=\overline{f}/\overline{s}$, for any $\overline{f}/\overline{s}\in\Hom_{\widetilde{\mathcal{B}}}(\overline{F(X)}, \overline{Y})$, where $\overline{f}\in \overline{\A}(\overline{\bullet},  \overline{Y})$ and $\overline{s}\in\overline{\I_{\B}}(\overline{\bullet},\overline{F(X)})$.\par
By the calculation, we have that
\begin{align*}
  \Psi_{\overline{X}, \overline{Y}}\Phi_{\overline{X}, \overline{Y}}(\overline{f}/\overline{s}) &=\Psi_{\overline{X}, \overline{Y}}(\overline{G(f)x}/\overline{y})  \\
   &=\overline{\varepsilon_{Y}FG(f)F(x)}/\overline{F(y)}
\end{align*}
where $\overline{x}$, $\overline{y}$ come from the following commutative diagram
$$ \xymatrix{ \overline{A}\ar@{=>}[d]_{\overline{y}}\ar@{}[dr]|-{\textrm{PB}}\ar[r]^{\overline{x}}&\overline{G(\bullet)}\ar@{=>}[d]^{\overline{G(s)}}\\
\overline{X}\ar[r]^-{\overline{\eta_{X}}}&\overline{GF(X)}}$$
For convenience, we put the following  commutative diagrams
\begin{align*}
\xymatrix{ \overline{FG(\bullet)}\ar@{=>}[d]_{\overline{\varepsilon_{\bullet}}}\ar[r]^{\overline{FG(f)}}&\overline{FG(Y)}\ar@{=>}[d]^{\overline{\varepsilon_{Y}}}\\
\overline{\bullet}\ar[r]^-{\overline{f}}&\overline{Y}}&&\xymatrix{ \overline{FG(\bullet)}\ar@{=>}[d]_{\overline{\varepsilon_{\bullet}}}\ar[r]^{\overline{FG(s)}}&\overline{FGF(X)}\ar@{=>}[d]^{\overline{\varepsilon_{F(X)}}}\\
\overline{\bullet}\ar[r]^-{\overline{s}}&\overline{F(X)}}
\end{align*}
Then, by Lemma~\ref{lem2-1} (3),  we have
\begin{align*}
\overline{\varepsilon_{F(X)}}\circ \overline{FG(s)} \circ \overline{F(x)}&=(\overline{\varepsilon_{F(X)}}\circ  \overline{F(\eta_{X})} ) \circ \overline{F(y)}=\overline{F(y)}\\
\overline{s}\circ \overline{\varepsilon_{\bullet}}\circ \overline{F(x)}&=\overline{\varepsilon_{F(X)}}\circ (\overline{FG(s)}\circ\overline{ F(x)})=(\overline{\varepsilon_{F(X)}} \circ  \overline{F(\eta_{X})} ) \circ \overline{F(y)}=\overline{F(y)}\\
\overline{\varepsilon_{Y}}\circ \overline{FG(f)} \circ \overline{F(x)}&=\overline{f}\circ\overline{\varepsilon_{\bullet}}\circ \overline{F(x)}.
\end{align*}
By the above equations, we have the following  commutative diagram
$$\xymatrix{&\overline{\bullet}\ar@{=>}[dl]_{\overline{s}}\ar[dr]^{\overline{f}}&\\
\overline{F(X)}&\overline{F(Z})\ar[u]^{\overline{v}}\ar[d]_{\overline{1}}\ar[r]\ar@{=>}[l]_-{\overline{F(y)}}&\overline{Y}\\
&\overline{F(Z)}\ar[ul]^{\overline{s'}}\ar[ur]_{\overline{f'}}&}$$
where $\overline{s'}=\overline{\varepsilon_{F(X)}}\circ \overline{FG(s)} \circ \overline{F(x)}$,
$\overline{f'}=\overline{\varepsilon_{Y}}\circ \overline{FG(f)} \circ \overline{F(x)}$ and $\overline{v}=\overline{\varepsilon_{\bullet}}\circ \overline{F(x)}$.
It implies that $\overline{f'}/\overline{s'}=\overline{f}/\overline{s}$. Hence, we have
\begin{align*}
  \Psi_{\overline{X}, \overline{Y}}\Phi_{\overline{X}, \overline{Y}}(\overline{f}/\overline{s}) &=\overline{\varepsilon_{Y}FG(f)F(x)}/\overline{F(y)}\\
  &=\overline{\varepsilon_{Y}FG(f)F(x)}/(\overline{\varepsilon_{F(X)}}\circ \overline{FG(f)} \circ \overline{F(x)})\\
  &=\overline{f'}/\overline{s'}\\
  &=\overline{f}/\overline{s}
\end{align*}
\textbf{Step 4.} Now, we check that $\Phi_{\overline{X}, \overline{Y}}\Psi_{\overline{X}, \overline{Y}}(\overline{g}/\overline{t})=\overline{g}/\overline{t}$, for any $\overline{g}/\overline{t}\in\Hom_{\widetilde{\mathcal{A}}}(\overline{X}, \overline{G(Y)}) $, where $\overline{g}\in \overline{\A}(\overline{\maltese},  \overline{G(Y)})$ and $\overline{t}\in\overline{\I_{\A}}(\overline{\maltese},\overline{X})$.\par
By the calculation, we have that
\begin{align*}
  \Phi_{\overline{X}, \overline{Y}}\Psi_{\overline{X}, \overline{Y}}(\overline{g}/\overline{t}) &=\Phi_{\overline{X}, \overline{Y}}(\overline{\varepsilon_{Y}F(g)}/\overline{F(t)})  \\
   &=\overline{G(\varepsilon_{Y})GF(g)x}/\overline{y}
\end{align*}
where $\overline{x}$, $\overline{y}$ come from the following commutative diagram
$$ \xymatrix{ \overline{A}\ar@{=>}[d]_{\overline{y}}\ar@{}[dr]|-{\textrm{PB}}\ar[r]^{\overline{x}}&\overline{GF(\maltese)}\ar@{=>}[d]^{\overline{GF(t)}}\\
\overline{X}\ar[r]^-{\overline{\eta_{X}}}&\overline{GF(X)}}$$
where $\eta_{X}$ is an unit map, $y\in \I_{\A}(A,X)$.
By Lemma~\ref{lem2-1} (3),  we have the following commutative diagram
$$\xymatrix{\overline{\maltese}\ar[d]_{\overline{\eta_{\maltese}}}\ar[r]^{\overline{g}}&\ar[d]_{\overline{\eta_{G(Y)}}}\overline{G(Y)}\ar[r]^{\overline{1}}&\overline{G(Y)}\ar[d]_{\overline{1}}\\
\overline{GF(\maltese)}\ar[r]^{\overline{GF(g)}}&\overline{GFG(Y)}\ar[r]^{\overline{G(\varepsilon_{Y})}}&\overline{G(Y)}}$$
Then we have that $\overline{g}=\overline{G(\varepsilon_{Y})}\circ \overline{GF(g)}\circ \overline{\eta_{\maltese}}$.\par
Since $\overline{\I_{\A}}$  is a multiplicative system, taking the homotopy pullback of $\overline{y}$ and $\overline{t}$, we get the following commutative diagram
$$\xymatrix{\overline{\clubsuit}\ar@{}[dr]|-{\textrm{PB}}\ar[r]^{\overline{u}}\ar@{=>}[d]_{\overline{s}}&\overline{A}\ar@{=>}[d]_{\overline{y}}\\
\overline{\maltese}\ar@{=>}[r]^{\overline{t}}&\overline{X}}$$
Then we can obtain the following diagram
$$\xymatrix{\overline{\clubsuit}\ar@/^3mm/[d]^{\overline{\eta_{\maltese}}\circ \overline{s}}\ar@/_3mm/[d]_{\overline{x}\circ \overline{u}}\\
\overline{GF(\maltese)}\ar@{=>}[d]^{\overline{GF(t)}}\\
\overline{GF(X)}}$$
Moreover, we have the equations $\overline{GF(t)}\circ\overline{\eta_{\maltese}}\circ \overline{s}=\overline{\eta_{X}}\circ \overline{t}\circ \overline{s}=\overline{GF(t)}\circ \overline{x}\circ \overline{u}$, which come from the following   commutative diagrams
 \begin{align*}
   \xymatrix{\overline{\maltese}\ar@{=>}[r]^{\overline{t}}\ar[d]^{\overline{\eta_{\maltese}}}&\overline{X}\ar[d]^{\overline{\eta_{X}}}\\
   \overline{GF(\maltese)}\ar@{=>}[r]^{\overline{GF(t)}}&\overline{GF(X)}} && \xymatrix{ \overline{\clubsuit}\ar[r]^{\overline{u}}\ar@{=>}[d]_{\overline{s}}&\overline{A}\ar@{=>}[d]_{\overline{y}}\ar[r]^{\overline{x}}&\overline{GF(\maltese)}\ar@{=>}[d]^{\overline{GF(t)}}\\
\overline{\maltese}\ar@{=>}[r]^{\overline{t}}&\overline{X}\ar[r]^-{\overline{\eta_{X}}}&\overline{GF(X)}}
 \end{align*}
Since  $\overline{\I_{\A}}$  is a multiplicative system, we get a morphism $\overline{k}\in\overline{\I_{\A}}$ such that there is a  diagram
 $$\xymatrix{\overline{\spadesuit}\ar@{=>}[d]^{\overline{k}}\\
 \overline{\clubsuit}\ar@/^3mm/[d]^{\overline{\eta_{\maltese}}\circ \overline{s}}\ar@/_3mm/[d]_{\overline{x}\circ \overline{u}}\\
GF(\maltese)}$$
with $\overline{\eta_{\maltese}}\circ \overline{s}\circ \overline{k}=\overline{x}\circ \overline{u}\circ \overline{k}$.\par
Now, we have the following equations
\begin{align*}
  &\Phi_{\overline{X}, \overline{Y}}\Psi_{\overline{X}, \overline{Y}}(\overline{g}/\overline{t})- \overline{g}/\overline{t}\\ &=\overline{G(\varepsilon_{Y})GF(g)x}/\overline{y}-\overline{g}/\overline{t}\\
  &=(\overline{G(\varepsilon_{Y})}\circ \overline{GF(g)}\circ \overline{x})/\overline{y}-(\overline{G(\varepsilon_{Y})}\circ \overline{GF(g)}\circ \overline{\eta_{\maltese}})/\overline{t}\\
  &=(\overline{G(\varepsilon_{Y})}\circ \overline{GF(g)}\circ \overline{x}\circ\overline{u}\circ \overline{k} )/(\overline{y}\circ\overline{u}\circ \overline{k})-(\overline{G(\varepsilon_{Y})}\circ \overline{GF(g)}\circ \overline{\eta_{\maltese}}\circ \overline{s}\circ \overline{k})/(\overline{t}\circ \overline{s}\circ \overline{k})\\
  &=0
\end{align*}
Thus, $\Phi_{\overline{X}, \overline{Y}}\Psi_{\overline{X}, \overline{Y}}(\overline{g}/\overline{t})=\overline{g}/\overline{t}$.\par
\textbf{Step 5.} It remains to prove that $\Psi$ is a natural transformation.\par
For any $\overline{f}/\overline{s}\in \Hom_{\widetilde{\A}}(\overline{X_{1}},\overline{X_{2}})$, we shall prove the following diagram is commutative
$$\xymatrix{\Hom_{\widetilde{\A}}(\overline{X_{2}},\overline{G(Y)})\ar[d]_{\Hom_{\widetilde{\A}}(\overline{f}/\overline{s},\overline{G(Y)})}\ar[rr]^{\Psi_{\overline{X_{2}}, \overline{Y}}}&&\Hom_{\widetilde{\B}}(\overline{F(X_{2})},\overline{Y})\ar[d]^{\Hom_{\widetilde{\B}}(\overline{F(f)}/\overline{F(s)},\overline{Y})}\\
\Hom_{\widetilde{\A}}(\overline{X_{1}},\overline{G(Y)})\ar[rr]^{\Psi_{\overline{X_{1}}, \overline{Y}}}&&\Hom_{\widetilde{\B}}(\overline{F(X_{1})},\overline{Y})}$$
Let $\overline{g}/\overline{t}\in\Hom_{\widetilde{\A}}(\overline{X_{2}},\overline{G(Y)})$. We have that
\begin{align*}
  \Psi_{\overline{X_{1}}, \overline{Y}}\Hom_{\widetilde{\A}}(\overline{f}/\overline{s},\overline{G(Y)})(\overline{g}/\overline{t}) &=\Psi_{\overline{X_{1}}, \overline{Y}}(\overline{g}/\overline{t}\circ\overline{f}/\overline{s}) \\
   &=\Psi_{\overline{X_{1}}, \overline{Y}}(\overline{gu}/\overline{sk})\\
   &=\overline{\varepsilon_{Y}F(g)F(u)}/\overline{F(s)F(k)}
\end{align*}
where $\overline{u}\in\overline{\A}$ and $\overline{k}\in\overline{\I_{\A}}$ come from the following commutative diagram
$$\xymatrix{&&\cdot\ar[dr]^-{\overline{u}}\ar@{=>}[dl]_{\overline{k}}&&\\
&\cdot\ar@{=>}[dl]_{\overline{s}}\ar[dr]^-{\overline{f}}&&\cdot\ar@{=>}[dl]_{\overline{t}}\ar[dr]^-{\overline{g}}&\\
\overline{X_{1}}&&\overline{X_{2}}&&\overline{G(Y)}}$$
Moreover, we have the equation $\overline{F(f)}\circ\overline{F(k)}=\overline{F(t)}\circ\overline{F(u)}$.
Hence, we have  the following commutative diagram
$$\xymatrix{&&\cdot\ar[dr]^-{\overline{F(u)}}\ar@{=>}[dl]_{\overline{F(k)}}&&\\
&\cdot\ar@{=>}[dl]_{\overline{F(s)}}\ar[dr]_-{\overline{F(f)}}&&\cdot\ar@{=>}[dl]^{\overline{F(t)}}\ar[dr]-<-0.5mm,-2.7mm>^-{\overline{\varepsilon_{Y}F(g)}}&\\
\overline{F(X_{1})}&&\overline{F(X_{2})}&&\overline{Y}}$$
Thus, we have
\begin{align*}
  \Psi_{\overline{X_{1}}, \overline{Y}}\Hom_{\widetilde{\A}}(\overline{f}/\overline{s},\overline{G(Y)})(\overline{g}/\overline{t})
   &=\overline{\varepsilon_{Y}F(g)F(u)}/\overline{F(s)F(k)}\\
   &=\overline{\varepsilon_{Y}F(g)}/\overline{F(t)}\circ \overline{F(f)}/\overline{F(s)}\\
   &=\Hom_{\widetilde{\B}}(\overline{F(f)}/\overline{F(s)},\overline{Y})\Psi_{\overline{X_{2}}, \overline{Y}}(\overline{g}/\overline{t}).
\end{align*}
Hence,  $\Psi$ is natural in the first variable.\par
For any $\overline{f}/\overline{s}\in\Hom_{\widetilde{\B}}(\overline{Y_{1}},\overline{Y_{2}})$, we also need to prove that the following diagram is  commutative.
$$\xymatrix{\Hom_{\widetilde{\A}}(\overline{X},\overline{G(Y_{1})})\ar[d]_{\Hom_{\widetilde{\A}}(\overline{X},\overline{G(f)}/\overline{G(s)})}\ar[rr]^{\Psi_{\overline{X}, \overline{Y_{1}}}}&&\Hom_{\widetilde{\B}}(\overline{F(X)},\overline{Y_{1}})\ar[d]^{\Hom_{\widetilde{\B}}(\overline{F(X)},\overline{f}/\overline{s})}\\
\Hom_{\widetilde{\A}}(\overline{X},\overline{G(Y_{2})})\ar[rr]^{\Psi_{\overline{X}, \overline{Y_{2}}}}&&\Hom_{\widetilde{\B}}(\overline{F(X)},\overline{Y_{2}})}$$
Let $\overline{g}/\overline{t}\in\Hom_{\widetilde{\A}}(\overline{X_{2}},\overline{G(Y)})$.
On one hand, by the definition, we have that
\begin{align*}
  \Psi_{\overline{X}, \overline{Y_{2}}}\Hom_{\widetilde{\A}}(\overline{X},\overline{G(f)}/\overline{G(s)})(\overline{g}/\overline{t}) & =\Psi_{\overline{X}, \overline{Y_{2}}}(\overline{G(f)}/\overline{G(s)}\circ\overline{g}/\overline{t}) \\
   & =\Psi_{\overline{X}, \overline{Y_{2}}}((\overline{G(f)}\circ\overline{u})/(\overline{t}\circ\overline{k}))\\
   &=(\overline{\varepsilon_{Y_{2}}}\circ\overline{FG(f)}\circ\overline{F(u)})/(\overline{F(t)}\circ\overline{F(k}))\\
   &=(\overline{f}\circ\overline{\varepsilon_{Z_{2}}}\circ\overline{F(u)})/(\overline{F(t)}\circ\overline{F(k}))
\end{align*}
where the composition $\overline{G(f)}/\overline{G(s)}\circ\overline{g}/\overline{t}=(\overline{G(f)}\circ\overline{u})/(\overline{t}\circ\overline{k})$ comes from the following commutative diagram,
\begin{equation}\label{dia2-1}
  \begin{split}
  \text{$\xymatrix{&&\overline{Z_{3}}\ar[dr]^-{\overline{u}}\ar@{=>}[dl]_{\overline{k}}&&\\
&\overline{Z_{1}}\ar@{=>}[dl]_{\overline{t}}\ar[dr]^-{\overline{g}}&&\overline{G(Z_{2})}\ar@{=>}[dl]_{\overline{G(s)}}\ar[dr]^-{\overline{G(f)}}&\\
\overline{X_{1}}&&\overline{G(Y_{1})}&&\overline{G(Y_{2})}}$}
\end{split}
\end{equation}

and the last equation comes from the following commutative square.
$$\xymatrix{\overline{FG(Z_{2})}\ar[d]_{\overline{\varepsilon_{Z_{2}}}}\ar[rr]^{\overline{FG(f)}}&&\overline{FG(Y_{2})}\ar[d]_{\overline{\varepsilon_{Y_{2}}}}\\
\overline{Z_{2}}\ar[rr]^{\overline{f}}&&\overline{Y_{2}}}$$

Moreover, from diagram (\ref{dia2-1}), we have the equation
\begin{equation}\label{eq2-0}
  \overline{FG(s)}\circ \overline{F(u)}=\overline{F(g)}\circ \overline{F(k)}
\end{equation}

On the other hand, by the definition, we have that
\begin{align*}
  \Hom_{\widetilde{\B}}(\overline{F(X)},\overline{f}/\overline{s})\Psi_{\overline{X}, \overline{Y_{1}}}(\overline{g}/\overline{t}) & =\Hom_{\widetilde{\B}}(\overline{F(X)},\overline{f}/\overline{s})((\overline{\varepsilon_{Y_{1}}}\circ\overline{F(g)})/\overline{F(t)}) \\
   & =\overline{f}/\overline{s}\circ\big((\overline{\varepsilon_{Y_{1}}}\circ\overline{F(g)})/\overline{F(t)}\big)\\
   &=(\overline{f}\circ \overline{u'})/(\overline{F(t)}\circ \overline{k'})
\end{align*}
where the last equation comes from the following commutative diagram
$$\xymatrix{&&\overline{Z_{4}}\ar[dr]^-{\overline{u'}}\ar@{=>}[dl]_{\overline{k'}}&&\\
&\overline{F(Z_{1})}\ar@{=>}[dl]_{\overline{F(t)}}\ar[dr]^-{\overline{\varepsilon_{Y_{1}}F(g)}}&&\overline{Z_{2}}\ar@{=>}[dl]_{\overline{s}}\ar[dr]^-{\overline{f}}&\\
\overline{F(X_{1})}&&\overline{Y_{1}}&&\overline{Y_{2}}}$$
Thus, we have the equation
\begin{equation}\label{eq2-1}
  \overline{s}\circ \overline{u'}=\overline{\varepsilon_{Y_{1}}}\circ\overline{F(g)} \circ \overline{k'}.
\end{equation}
By the natural transformation $\varepsilon:FG\rightarrow id$, we have the equation
\begin{equation}\label{eq2-2}
\overline{s}\circ\overline{\varepsilon_{Z_{2}}}=\overline{\varepsilon_{Y_{1}}}\circ \overline{FG(s)}.
\end{equation}\par
Since $\overline{\I_{\B}}$ is  a multiplicative system, we take the homotopy pullback of $\overline{F(k)}$ and $\overline{k'}$. Then we obtain the following commutative square
$$\xymatrix{\overline{\maltese}\ar@{}[dr]|-{\textrm{PB}}\ar@{=>}[d]_-{\overline{v}}\ar[r]^-{\overline{w}}&\overline{F(Z_{3})\ar@{=>}[d]^{\overline{F(k)}}}\\
\overline{Z_{4}}\ar@{=>}[r]^-{k'}&\overline{F(Z_{1})}}$$
By the above square and equations (\ref{eq2-0}), (\ref{eq2-1}), and (\ref{eq2-2}), we have the following diagram
$$\xymatrix{\overline{\maltese}\ar@/^3mm/[d]^{\overline{\varepsilon_{Z_{2}}}\circ \overline{F(u)} \circ \overline{w}}\ar@/_3mm/[d]_{\overline{u'}\circ \overline{v}}\\
\overline{Z_{2}}\ar@{=>}[d]^{\overline{s}}\\
\overline{Y_{1}}}$$
such that
\begin{align*}
  \overline{s}\circ\overline{\varepsilon_{Z_{2}}}\circ \overline{F(u)} \circ \overline{w} &= \overline{\varepsilon_{Y_{1}}}\circ \overline{FG(s)}\circ \overline{F(u)} \circ \overline{w} \\
   & =\overline{\varepsilon_{Y_{1}}}\circ \overline{F(g)}\circ \overline{F(k)} \circ \overline{w}\\
   &=\overline{\varepsilon_{Y_{1}}}\circ \overline{F(g)}\circ \overline{k'} \circ \overline{v}\\
   &=\overline{s}\circ \overline{u'}\circ \overline{v}
\end{align*}
Since $\overline{\I_{\B}}$ is  a multiplicative system, there exists a morphism $\overline{h}\in\overline{\I_{\B}}(\overline{\spadesuit},\overline{\maltese})$ such that the following diagram
$$\xymatrix{\overline{\spadesuit}\ar@{=>}[d]^{\overline{h}}\\
\overline{\maltese}\ar@/^3mm/[d]^{\overline{\varepsilon_{Z_{2}}}\circ \overline{F(u)} \circ \overline{w}}\ar@/_3mm/[d]_{\overline{u'}\circ \overline{v}}\\
\overline{Z_{2}}}$$
satisfies
\begin{equation}\label{eq2-3}
  \overline{\varepsilon_{Z_{2}}}\circ \overline{F(u)} \circ \overline{w}\circ \overline{h}=\overline{u'}\circ \overline{v}\circ \overline{h}.
\end{equation}
At last, we show that $\Psi_{\overline{X}, \overline{Y_{2}}}\Hom_{\widetilde{\A}}(\overline{X},\overline{G(f)}/\overline{G(s)})(\overline{g}/\overline{t})=\Hom_{\widetilde{\B}}(\overline{F(X)},\overline{f}/\overline{s})\Psi_{\overline{X}, \overline{Y_{1}}}(\overline{g}/\overline{t})$.\par
Indeed, by equation (\ref{eq2-3}), we see that
\begin{align*}
   & \Psi_{\overline{X}, \overline{Y_{2}}}\Hom_{\widetilde{\A}}(\overline{X},\overline{G(f)}/\overline{G(s)})(\overline{g}/\overline{t})-\Hom_{\widetilde{\B}}(\overline{F(X)},\overline{f}/\overline{s})\Psi_{\overline{X}, \overline{Y_{1}}}(\overline{g}/\overline{t}) \\
   &= (\overline{f}\circ\overline{\varepsilon_{Z_{2}}}\circ\overline{F(u)})/(\overline{F(t)}\circ\overline{F(k}))-(\overline{f}\circ \overline{u'})/(\overline{F(t)}\circ \overline{k'})\\
   &=(\overline{f}\circ\overline{\varepsilon_{Z_{2}}}\circ\overline{F(u)}\circ \overline{w}\circ \overline{h})/(\overline{F(t)}\circ\overline{F(k})\circ \overline{w}\circ \overline{h})-(\overline{f}\circ \overline{u'}\circ \overline{v}\circ \overline{h})/(\overline{F(t)}\circ \overline{k'}\circ \overline{v}\circ \overline{h})\\
   &=(\overline{f}\circ\overline{\varepsilon_{Z_{2}}}\circ\overline{F(u)}\circ \overline{w}\circ \overline{h})/(\overline{F(t)}\circ\overline{k'}\circ \overline{v}\circ \overline{h})-(\overline{f}\circ \overline{u'}\circ \overline{v}\circ \overline{h})/(\overline{F(t)}\circ \overline{k'}\circ \overline{v}\circ \overline{h})\\
   &=(\overline{f}\circ \overline{u'}\circ \overline{v}\circ \overline{h})/(\overline{F(t)}\circ \overline{k'}\circ \overline{v}\circ \overline{h})-(\overline{f}\circ \overline{u'}\circ \overline{v}\circ \overline{h})/(\overline{F(t)}\circ \overline{k'}\circ \overline{v}\circ \overline{h})\\
   &=0
\end{align*}
Hence,  $\Psi$ is natural in the second variable.\par
This completes the proof.
\end{proof}

\begin{corollary}\label{cor2-2}
Keep the same notions and assumptions as {\rm Theorem \ref{prop2-1}}. The following hold.
\begin{enumerate}
  \item If $G$ is fully faithful, then so does $\widetilde{G}$.
  \item If $F$ is fully faithful, then so does $\widetilde{F}$.
\end{enumerate}
\end{corollary}
\begin{proof}
(1) It suffices to show that the counit map $\widetilde{\varepsilon}:\widetilde{F}\widetilde{G}\rightarrow id$ is a natural isomorphism.\par
By the definition of $\widetilde{\varepsilon}$, for any $\overline{Y}\in \widetilde{\B}$, we have $\widetilde{\varepsilon}_{\overline{Y}}=\Phi_{\overline{G(Y)}, \overline{Y}}^{-1}(\overline{id_{G(Y)}}/\overline{id_{G(Y)}})=\overline{\varepsilon_{Y}}/\overline{id_{FG(Y)}}$.
Indeed, this comes from the following commutative
$$\xymatrix{ &\overline{G(Y)}\\ \overline{G(Y)}\ar[ur]^{\overline{f}}\ar@{=>}[d]_{\overline{id_{G(Y)}}}\ar@{}[dr]|-{\textrm{PB}}\ar[r]^{\overline{\eta_{G(Y)}}}&\overline{GFG(Y)}\ar@{=>}[d]^{\overline{G(id_{FG(Y)})}}\ar[u]_{\overline{G(\varepsilon_{Y})}}\\
\overline{G(Y)}\ar[r]^-{\overline{\eta_{G(Y)}}}&\overline{GFG(Y)}}$$
where $\overline{f}=\overline{G(\varepsilon_{Y})}\circ\overline{\eta_{G(Y)}}=\overline{id_{G(Y)}}$ by Lemma~\ref{lem2-1} (3). Since $G$ is fully faithful, $\varepsilon_{Y}$ is an isomorphism. Denote by $\varepsilon_{Y}^{-1}$ the inverse of $\varepsilon_{Y}$.\par
 Let $\widetilde{\varepsilon}_{\overline{Y}}^{-1}=\overline{\varepsilon_{Y}^{-1}}/\overline{id_{Y}}$.
 We can check that $\widetilde{\varepsilon}_{\overline{Y}}\circ\widetilde{\varepsilon}_{\overline{Y}}^{-1}=\overline{id_{FG(Y)}}/\overline{id_{FG(Y)}}$ and $\widetilde{\varepsilon}_{\overline{Y}}^{-1}\circ\widetilde{\varepsilon}_{\overline{Y}}=\overline{id_{Y}}/\overline{id_{Y}}$.
Thus, the counit map $\widetilde{\varepsilon}:\widetilde{F}\widetilde{G}\rightarrow id$ is a natural isomorphism.\par
(2) It suffices to show that the unit map $\widetilde{\eta}:id\rightarrow\widetilde{G}\widetilde{F}$ is a natural isomorphism.\par
By the definition of $\widetilde{\eta}$, for any $\overline{X}\in \widetilde{\A}$, we have $\widetilde{\eta}_{\overline{X}}=\Phi_{\overline{X}, \overline{F(X)}}(\overline{id_{F(X)}}/\overline{id_{F(X)}})=\overline{\eta_{X}}/\overline{id_{X}}$.
Indeed, this comes from the following commutative
$$\xymatrix{ &\overline{GF(X)}\\ \overline{X}\ar[ur]^{\overline{f}}\ar@{=>}[d]_{\overline{id_{X}}}\ar@{}[dr]|-{\textrm{PB}}\ar[r]^{\overline{\eta_{X}}}&\overline{G F(X)}\ar@{=>}[d]^{\overline{G(id_{F(X)})}}\ar[u]_{\overline{G(id_{F(X)})}}\\
\overline{X}\ar[r]^-{\overline{\eta_{X}}}&\overline{GF(X)}}$$
where $\overline{f}=\overline{G(id_{F(X)})}\circ\overline{\eta_{X}}=\overline{\eta_{X}}$ by Lemma~\ref{lem2-1} (3). Since $F$ is fully faithful, $\eta_{X}$ is an isomorphism. Denote by $\eta_{X}^{-1}$ the inverse of $\eta_{X}$.\par
 Let $\widetilde{\eta}_{\overline{X}}^{-1}=\overline{\eta_{X}^{-1}}/\overline{id_{GF(X)}}$.
 We can check that $\widetilde{\eta}_{\overline{X}}\circ\widetilde{\eta}_{\overline{X}}^{-1}=\overline{id_{G F(X)}}/\overline{id_{G F(X)}}$ and $\widetilde{\eta}_{\overline{X}}^{-1}\circ\widetilde{\eta}_{\overline{X}}=\overline{id_{X}}/\overline{id_{X}}$.
Thus, the unit map $\widetilde{\eta}:id\rightarrow\widetilde{G}\widetilde{F}$ is a natural isomorphism.
\end{proof}
\section{$\s$-recollements and its properties}
\begin{definition}
Let $\mathcal{A}$, $\mathcal{B}$ and $\mathcal{C}$ be three extriangulated categories. A weak \emph{$\s$-recollement} of $\mathcal{B}$ relative to
$\mathcal{A}$ and $\mathcal{C}$, denoted by ($\mathcal{A}$, $\mathcal{B}$, $\mathcal{C}$), is a diagram of additive functors
\begin{equation}\label{recolle}
  \xymatrix{\mathcal{A}\ar[rr]|{i_{*}}&&\ar@/_1pc/[ll]|{i^{*}}\ar@/^1pc/[ll]|{i^{!}}\mathcal{B}
\ar[rr]|{j^{\ast}}&&\ar@/_1pc/[ll]|{j_{!}}\ar@/^1pc/[ll]|{j_{\ast}}\mathcal{C}}
\end{equation}
 which satisfies the following conditions:
\begin{itemize}
  \item [(SR1)] $(i^{*}, i_{\ast}, i^{!})$ and $(j_!, j^\ast, j_\ast)$ are adjoint triples, where $i_{\ast}$ and $j^{\ast}$ are exact functors.
  \item [(SR2)] $ j^{\ast}i_{\ast}=0$.
  \item [(SR3)] $i_\ast$, $j_!$ and $j_\ast$ are fully faithful.
  \item [(SR4)] For each $X\in\mathcal{B}$, there exist a  left $\s$-exact sequence
  \begin{equation}\label{first}
  \xymatrix{i_\ast i^! X\ar[r]^-{\theta_X}&X\ar[r]^-{\vartheta_X}&j_\ast j^\ast X&}
   \end{equation}
  \\
  and  a  right $\s$-exact  sequence
  \begin{equation}\label{second}
  \xymatrix{&j_! j^\ast X\ar[r]^-{\upsilon_X}&X\ar[r]^-{\nu_X}&i_\ast i^\ast X&}
   \end{equation}
 with $\textrm{Cone}(\theta_X)\in \Ker i^!$ and $\textrm{CoCone}(\nu_X)\in \Ker i^\ast$, where all maps in these sequences are given by the adjunction morphisms.
\end{itemize}
 A weak $\s$-recollement ($\mathcal{A}$, $\mathcal{B}$, $\mathcal{C}$) is said to be an \emph{$\s$-recollement} if moreover it  satisfies the following condition:
 \begin{itemize}
 \item [(SR5)] {\rm$\Ker i^{!}\bigcap\Ker j^{\ast}=0$} or {\rm$\Ker i^\ast\bigcap \Ker j^{\ast}=0$}.
 \end{itemize}
\end{definition}

\begin{lemma}\label{lem3-1}
Let {\rm($\mathcal{A}$, $\mathcal{B}$, $\mathcal{C}$)} be a weak $\s$-recollement of extriangulated categories. Then the sequences {\rm(\ref{first})} and {\rm(\ref{second})} can be extended to the following 4-terms $\s$-exact sequences
\begin{align*}
 \xymatrix{i_\ast i^! X\ar[r]^-{\theta_X}&X\ar[r]^-{\vartheta_X}&j_\ast j^\ast X\ar[r]&Z} \\
 \xymatrix{ K\ar[r]&j_! j^\ast X\ar[r]^-{\upsilon_X}&X\ar[r]^-{\nu_X}&i_\ast i^\ast X}
\end{align*}
with $K$ and $Z$ lie in {\rm$\Ker j^{\ast}$}.
\end{lemma}
\begin{proof}
It is enough to show the existence of the first 4-terms $\s$-exact sequence since the second 4-terms $\s$-exact sequence can be obtained by similar argument. By (SR4), there exists a left $\s$-exact sequence
  \begin{equation*}
  \xymatrix{i_\ast i^! X\ar[r]^-{\theta_X}&X\ar[r]^-{\vartheta_X}&j_\ast j^\ast X&}
   \end{equation*}
   with $\textrm{Cone}(\theta_X)\in \ker i^!$. Then we have two $\E_{\B}$-triangles
   \begin{align*}
    i_\ast i^! X\xrightarrow{\theta_X} X\xrightarrow{f_{1}} \textrm{Cone}(\theta_X)\dashrightarrow \\
     \textrm{Cone}(\theta_X)\xrightarrow{f_{2}} j_\ast j^\ast X\rightarrow Z\dashrightarrow.
   \end{align*}
   Applying the exact functor $j^{\ast}$ to the above two $\E_{\B}$-triangles, we have the following  $\E_{\C}$-triangles
   \begin{align*}
    &j^{\ast}i_\ast i^! X\xrightarrow{j^{\ast}(\theta_X)}j^{\ast} X\xrightarrow{j^{\ast}(f_{1})} j^{\ast}(\textrm{Cone}(\theta_X))\dashrightarrow \\
    & j^{\ast}(\textrm{Cone}(\theta_X))\xrightarrow{j^{\ast}(f_{2})} j^{\ast}j_\ast j^\ast X\rightarrow j^{\ast} Z\dashrightarrow.
   \end{align*}
   Since $ j^{\ast}i_\ast i^! X=0$, we know that $j^{\ast}(f_{1})$ is an isomorphism. Note that $j_\ast$ is fully faithful. It yields that $j^{\ast}(\vartheta_X)$ is an isomorphism. It follows that $j^{\ast}(f_{2})$ is an isomorphism since $j^{\ast}(\vartheta_X)=j^{\ast}(f_{2})j^{\ast}(f_{1})$. Therefore, $j^{\ast} Z=0$. It completes the proof. Similarly, we can obtain the second sequence.
\end{proof}

\begin{lemma}\label{lem3-2}
Let {\rm($\mathcal{A}$, $\mathcal{B}$, $\mathcal{C}$)} be an $\s$-recollement of extriangulated categories. Then {\rm$\Ker j^{\ast}=\Im i_{\ast}$}.
\end{lemma}
\begin{proof} By (NER2), we know that $\Im i_{\ast}\subseteq \Ker j_{\ast}$. Assume that $X\in \Ker j_{\ast}$. Then by (SR4) and Lemma \ref{lem3-1}, there exist two 4-terms $\s$-exact sequences:
\begin{align}
 \xymatrix{i_\ast i^! X\ar[r]^-{\theta_X}&X\ar[r]^-{}&0\ar[r]&Z} \label{eq3-1}\\
 \xymatrix{ K\ar[r]&0\ar[r]^-{}&X\ar[r]^-{\nu_X}&i_\ast i^\ast X} \label{eq3-2}
\end{align}
with $K$ and $Z$ lie in {\rm$\Ker j^{\ast}$}, $\textrm{Cone}(\theta_X)\in \Ker i^!$ and $\textrm{CoCone}(\nu_X)\in \Ker i^\ast$.
By  $\s$-exact sequences  (\ref{eq3-1}) and (\ref{eq3-2}), we have the following
$\E_{\B}$-triangles
   \begin{align*}
     &\textrm{Cone}(\theta_X)\xrightarrow{f_{2}}0\rightarrow Z\dashrightarrow\\
     &K\rightarrow 0 \rightarrow\textrm{CoCone}(\nu_X)\dashrightarrow.
   \end{align*}
   Applying  the exact functor $j^{\ast}$ to the above two $\E_{\B}$-triangles, we have the following  $\E_{\C}$-triangles
   \begin{align*}
     &j^{\ast}(\textrm{Cone}(\theta_X))\xrightarrow{j^{\ast}(f_{2})}0\rightarrow j^{\ast}(Z)\dashrightarrow\\
     &j^{\ast}(K)\rightarrow 0 \rightarrow j^{\ast}(\textrm{CoCone}(\nu_X))\dashrightarrow.
   \end{align*}
   Since $K$ and $Z$ lie in {\rm$\Ker j^{\ast}$}, $j^{\ast}(Z)=0=j^{\ast}(K)=0$. Thus, $ j^{\ast}(\textrm{Cone}(\theta_X))=0=j^{\ast}(\textrm{CoCone}(\nu_X))$. That is $\textrm{Cone}(\theta_X)\in \Ker i^!\bigcap \Ker j^{\ast}$ and $\textrm{CoCone}(\nu_X)\in  \Ker i^\ast\bigcap \Ker j^{\ast}$. By  (SR5), we have that $\textrm{Cone}(\theta_X)=0$ or $\textrm{CoCone}(\nu_X)=0$. Suppose that $\textrm{Cone}(\theta_X)=0$. From the $\E_{\B}$-triangle $i_\ast i^! X\xrightarrow{\theta_X} X\xrightarrow{} \textrm{Cone}(\theta_X)\dashrightarrow$, we have the isomorphism $X\cong i_\ast i^! X$. Thus, $X\in\Im i_\ast $. By similar arguments, we get the desired result when $\textrm{CoCone}(\nu_X)=0$.   It completes the proof.
\end{proof}
\begin{proposition}\label{prop3-4}
Let {\rm($\mathcal{A}$, $\mathcal{B}$, $\mathcal{C}$)} be a weak $\s$-recollement of extriangulated categories. If $\mathcal{A}$, $\mathcal{B}$, $\mathcal{C}$ are abelian categories, then {\rm($\mathcal{A}$, $\mathcal{B}$, $\mathcal{C}$)} is a recollement of abelian categories.
\end{proposition}
\begin{proof} By Lemma \ref{lem3-2}, it suffices to show that in this case, weak $\s$-recollement satisfies (SR5). For any $X\in \Ker i^!\bigcap \Ker j^{\ast}$, by Lemma \ref{lem3-1}, we get an exact sequences
\begin{align*}
 \xymatrix{0\ar[r]&i_\ast i^! X\ar[r]^-{\theta_X}&X\ar[r]^-{\vartheta_X}&j_\ast j^\ast X\ar[r]&Z\ar[r]&0.}
\end{align*}
Since $i_\ast i^! X=0=j_\ast j^\ast X$, be the exactness,  we know that $X=0$. Similarly, we can prove that $\Ker i^\ast\bigcap \Ker j^{\ast}=0$. It completes the proof.
\end{proof}

\begin{lemma}\label{lemm3-1.1}
Let {\rm($\mathcal{A}$, $\mathcal{B}$, $\mathcal{C}$)} be  a  weak $\s$-recollement. Then the following statements hold.
\begin{enumerate}
  \item $i^{\ast}j_{!}=0$;
  \item $i^{!}j_{\ast}=0$.
\end{enumerate}
\end{lemma}
\begin{proof}
(1) Let $X\in\C$. Since $j^{\ast}i_{\ast}=0$, we have the following
\begin{align*}
  \Hom_{\C}(i^{\ast}j_{!}(X),i^{\ast}j_{!}(X)) &\cong \Hom_{\B}(j_{!}(X),i_{\ast}i^{\ast}j_{!}(X)) \\
   &\cong \Hom_{\B}(X,j^{*}i_{\ast}i^{\ast}j_{!}(X))\\
   &=0.
\end{align*}
Thus, $i^{\ast}j_{!}=0$. Similarly, one can prove (2).
\end{proof}
\begin{proposition}\label{prop3-3}
Let {\rm($\mathcal{A}$, $\mathcal{B}$, $\mathcal{C}$)} be an $\s$-recollement of extriangulated categories and $X\in\mathcal{B}$. Then the following statements hold.
\begin{enumerate}
  \item If $i^{!}$ is exact, there exists an $\mathbb{E}_\mathcal{B}$-triangle
  \begin{equation*}\label{third}
  \xymatrix{i_\ast i^! X\ar[r]^-{\theta_X}&X\ar[r]^-{\vartheta_X}&j_\ast j^\ast X\ar@{-->}[r]&}
   \end{equation*}
 where $\theta_X$ and  $\vartheta_X$ are given by the adjunction morphisms.
  \item If $i^{\ast}$ is exact, there exists an $\mathbb{E}_\mathcal{B}$-triangle
  \begin{equation*}\label{four}
  \xymatrix{ j_! j^\ast X\ar[r]^-{\upsilon_X}&X\ar[r]^-{\nu_X}&i_\ast i^\ast X \ar@{-->}[r]&}
   \end{equation*}
where $\upsilon_X$ and $\nu_X$ are given by the adjunction morphisms.

\end{enumerate}

\end{proposition}
\begin{proof}
 (1) Assume that $i^!$ is exact. By (SR4) and Lemma \ref{lem3-1}, there exists an   $\s$-exact sequence
  \begin{equation*}
  \xymatrix{i_\ast i^! X\ar[r]^-{\theta_X}&X\ar[r]^-{\vartheta_X}&j_\ast j^\ast X\ar[r]&Z}
   \end{equation*}
   with $\textrm{Cone}(\theta_X)\in \ker i^!$ and $Z\in\Ker j^{\ast}$. Then we have two $\E_{\B}$-triangles
   \begin{align*}
    i_\ast i^! X\xrightarrow{\theta_X} X\xrightarrow{f_{1}} \textrm{Cone}(\theta_X)\dashrightarrow \\
     \textrm{Cone}(\theta_X)\xrightarrow{f_{2}} j_\ast j^\ast X\rightarrow Z\dashrightarrow.
   \end{align*}
   Applying $i^!$ to the bottom  $\E_{\B}$-triangle, we get the following  $\E_{\C}$-triangle
   \begin{align*}
    i^! \textrm{Cone}(\theta_X)\xrightarrow{j^\ast(f_{2})} i^! j_\ast j^\ast X\rightarrow i^! Z\dashrightarrow.
   \end{align*}
   Since $ \textrm{Cone}(\theta_X)\in \ker i^!$ and $i^! j_\ast j^\ast X=0$ ( by Lemma \ref{lemm3-1.1} (2)), we have $i^! Z=0$. By Lemma \ref{lem3-2}, there exists $Z'\in \A$ such that $Z\cong i_{\ast}(Z')$. Since $i_{\ast}$ is fully faithful, $Z'\cong i^!i_{\ast}(Z')\cong i^! Z=0$. Thus, $Z\cong 0$. It implies that $f_{2}:\textrm{Cone}(\theta_X)\rightarrow j_\ast j^\ast X$ is  an isomorphism. Then we get the  desired $\E_{\B}$-triangle $i_\ast i^! X\xrightarrow{\theta_X} X\rightarrow j_\ast j^\ast X\dashrightarrow$.\par
   (2) It follows from a similar argument as in the statement (1).
\end{proof}
\begin{corollary}\label{cor3-1}
Let {\rm($\mathcal{A}$, $\mathcal{B}$, $\mathcal{C}$)} be an $\s$-recollement of extriangulated categories and $X\in\mathcal{B}$. If $\mathcal{A}$, $\mathcal{B}$, and $\mathcal{C}$ are triangulated categories, then {\rm ($\mathcal{A}$, $\mathcal{B}$, $\mathcal{C}$)} is a recollement of triangulated categories.
\end{corollary}
\begin{proof} Since $i_{\ast}$ and $j^{\ast}$ are exact, by  \cite[Lemma 5.3.6]{N}, we have that $i^\ast$ and $i^!$ are also exact. By Proposition \ref{prop3-3}, ($\mathcal{A}$, $\mathcal{B}$, $\mathcal{C}$) is a recollement of triangulated categories.
\end{proof}

\textbf{WIC Condition} Let $\mathcal{A}$ be an extriangulated category.
\begin{enumerate}
\item Let $f : X\rightarrow Y$ and $g : Y \rightarrow Z$ be any composable pair of
morphisms in $\A$. If $gf$ is an inflation, then $f$ is an inflation.
\item Let $f : X\rightarrow Y$ and $g : Y \rightarrow Z$ be any composable pair of morphisms in $\A$. If $gf$ is a deflation, then $g$ is a deflation.
\end{enumerate}

\begin{definition}[\cite{WWZ}]
A morphism $f$ in $\mathcal{C}$ is called {\em compatible}, if ``$f$ is both an inflation and a deflation" implies that $f$ is an isomorphism.
\end{definition}
It is clear that all morphisms are compatible in an exact category. While, the compatible morphisms in a triangulated category $\C$ are just the isomorphisms in $\C$.

Now, we will remark that each recollement of extriangulated categories in the sense of \cite[Definition 3.1]{WWZ} is an $\s$-recollement under (WIC) Condition.\\

\begin{proposition} Under {\rm(WIC)} Conditions,
each recollement of extriangulated categories in the sense of {\rm\cite[Definition 3.1]{WWZ}} is an $\s$-recollement.
\end{proposition}
\begin{proof} Let ($\mathcal{A}$, $\mathcal{B}$, $\mathcal{C}$) be a recollement of extriangulated categories in the sense of \cite[Definition 3.1]{WWZ}.
By Definition 3.1 in the reference \cite{WWZ}, for any $X\in \B$, there exist a a left exact $\mathbb{E}_{\mathcal B}$-triangle sequence
\begin{equation*}
\xymatrix{i_\ast i^! X\ar[r]^-{\theta_X}&X\ar[r]^-{\vartheta_X}&j_\ast j^\ast X\ar[r]&i_\ast A}
\end{equation*}
 and a right exact $\mathbb{E}_{\mathcal B}$-triangle sequence
\begin{equation*}
\xymatrix{i_\ast\ar[r] A' &j_! j^\ast X\ar[r]^-{\upsilon_X}&X\ar[r]^-{\nu_X}&i_\ast i^\ast X&}
\end{equation*}
with $A$ and $A'$ lie in $\A$. Then there are four $\mathbb{E}_{\mathcal B}$-triangles:
\begin{align}
  &i_\ast i^! X\xrightarrow{\theta_X}X\rightarrow \textrm{Cone}(\theta_X)\dashrightarrow \label{eq3-5}\\
&\textrm{Cone}(\theta_X)\xrightarrow{f} j_\ast j^\ast X\rightarrow i_\ast A\dashrightarrow \label{eq3-3}\\
&i_\ast A'\rightarrow j_! j^\ast X\xrightarrow{g} \textrm{CoCone}(\nu_X)\dashrightarrow\label{eq3-4}\\
&\textrm{CoCone}(\nu_X)\rightarrow X\xrightarrow{\nu_X} i_\ast i^\ast X\dashrightarrow\label{eq3-6}
\end{align}
where $f$ and $g$ are compatible morphism.\par
 It suffices  to show that $\textrm{Cone}(\theta_X)\in \Ker i^!$,  $\textrm{CoCone}(\nu_X)\in \Ker i^\ast$, and $\Ker i^{!}\bigcap\Ker j^{\ast}=0$ (or {\rm$\Ker i^\ast\bigcap \Ker j^{\ast}=0$}). \par
 Applying the left exact functor $i^!$ to the $\mathbb{E}_{\mathcal B}$-triangle (\ref{eq3-3}), we get a  $\mathbb{E}_{\mathcal C}$-triangle $$i^!(\textrm{Cone}(\theta_X))\xrightarrow{i^!(f)} i^!j_\ast j^\ast X\rightarrow K\dashrightarrow.$$
Since $ i^!j_\ast j^\ast X=0$ and $i^!(f)$ is compatible morphism, $i^!(\textrm{Cone}(\theta_X))=0$ and so, $\textrm{Cone}(\theta_X)\in \Ker i^!$. \par
Applying the right exact functor $i^\ast$ to the $\mathbb{E}_{\mathcal B}$-triangle (\ref{eq3-4}), we get a  $\mathbb{E}_{\mathcal C}$-triangle
$$W\rightarrow i^\ast j_! j^\ast X\xrightarrow{i^\ast(g)} i^\ast(\textrm{CoCone}(\nu_X))\dashrightarrow.$$
Since $i^\ast j_! j^\ast X=0$ and  $i^\ast(g)$ is  compatible morphism, $i^\ast(\textrm{CoCone}(\nu_X))=0$ and so,  $\textrm{CoCone}(\nu_X)\in\Ker i^\ast$.\par
For any $Y\in\Ker i^{!}\bigcap\Ker j^{\ast}$, then from (\ref{eq3-5}) and (\ref{eq3-3}), $Y\cong \textrm{Cone}(\theta_Y)\cong 0$ since the compatible morphism $f:\textrm{Cone}(\theta_Y)\rightarrow 0$ is both inflation and deflation.  By similar argument, we have $\Ker i^\ast\bigcap \Ker j^{\ast}=0$.
Therefore, the recollement ($\mathcal{A}$, $\mathcal{B}$, $\mathcal{C}$) is  an $\s$-recollement.
\end{proof}
\begin{proposition}\label{prop3-6}
Let {\rm($\mathcal{A}$, $\mathcal{B}$, $\mathcal{C}$)} be  an $\s$-recollement. Then the following statements hold.
\begin{enumerate}
  \item If  $i^{\ast}$  is exact, then {\rm$\Ker i^{\ast}=\Im j_{!}$}.
  \item If  $i^{!}$ is  exact, then {\rm$\Ker i^{!}=\Im j_{\ast}$}.
  \item If $i^{\ast}$ is exact and $j_{!}$ is right $\s$-exact, then $j_{!}$ is exact.
  \item If $i^{!}$ is exact and $j_{\ast}$ is left $\s$-exact, then $j_{\ast}$ is exact.
\end{enumerate}
\end{proposition}
\begin{proof}

We assume that ($\mathcal{A}$, $\mathcal{B}$, $\mathcal{C}$) is  an $\s$-recollement. Here, we just prove (1) and (3), the proofs of (2) and (4) are similar to (1) and (3), respectively.\par
(1) Assume that $i^{\ast}$  is exact. For any $X\in \Ker i^{\ast}$, by Lemma \ref{lem3-1} and Lemma \ref{lem3-2}, we have an $\E_{\B}$-triangle
$$i_{\ast}(A)\rightarrow j_{!}j^{\ast}(X)\rightarrow X\dashrightarrow.$$
Applying $i^{\ast}$ to the above $\E_{\B}$-triangle, one obtain the following  $\E_{\A}$-triangle
$$i^{\ast}i_{\ast}(A)\rightarrow i^{\ast}j_{!}j^{\ast}(X)\rightarrow i^{\ast}(X)\dashrightarrow.$$
Since $i^{\ast}j_{!}j^{\ast}(X)=i^{\ast}(X)=0$ and $i_{\ast}$ is fully faithful, $A\cong i^{\ast}i_{\ast}(A)\cong0$. Then $i_{\ast}(A)=0$ and hence, $X \cong j_{!}j^{\ast}(X)$. It means that $\Ker i^{\ast}\subseteq \Im j_{!}$. By Lemma \ref{lemm3-1.1} (1), we know that $\Ker i^{\ast}=\Im j_{!}$.\par

(3) It suffices to show that for any $\E_{\C}$-triangle $A\overset{x}{\longrightarrow}B\overset{y}{\longrightarrow}C\dashrightarrow$ , the sequence  $$j_{!}(A)\overset{j_{!}(x)}{\longrightarrow}j_{!}(B)\overset{j_{!}(y)}{\longrightarrow}j_{!}(C)$$
is a conflation in $\B$.\par
Since $j_{!}$ is right exact, we have a right $\s$-exact sequence
$$j_{!}(A)\overset{j_{!}(x)}{\longrightarrow}j_{!}(B)\overset{j_{!}(y)}{\longrightarrow}j_{!}(C)$$
Then there are two  $\E_{\B}$-triangles
\begin{align}
  &K\rightarrow j_{!}(A)\xrightarrow{b} Q\dashrightarrow \label{eq3-7}\\
  &Q\xrightarrow{a} j_{!}(B)\xrightarrow{j_{!}(y)}  j_{!}(C)\dashrightarrow\label{eq3-8}
\end{align}
where $j_{!}(x)=ab$. Applying $i^{\ast}$ to $\E_{\B}$-triangle (\ref{eq3-8}), from the exactness of $i^{\ast}$, we have an $\E_{\A}$-triangle $i^{\ast}(Q)\xrightarrow{i^{\ast}(a)} i^{\ast}j_{!}(B)\xrightarrow{i^{\ast}j_{!}(y)}  i^{\ast}j_{!}(C)\dashrightarrow$. By Lemma \ref{lemm3-1.1} (1), we know that $i^{\ast}j_{!}=0$ and hence, $i^{\ast}(Q)=0$. Then, for the object $Q$,  by Proposition \ref{prop3-3} (2),  the counit map $\varepsilon_{Q}:j_! j^\ast (Q)\rightarrow Q$ is an isomorphism. \par
Applying the exact $j^{\ast}$ to $\E_{\B}$-triangle (\ref{eq3-8}), we obtain a morphism $\alpha:X\rightarrow j^\ast (Q)$ such that the following diagram commutative
$$\xymatrix{
A\ar[r]^{x}\ar[d]^{\alpha}&B\ar[d]^{\eta_{B}}\ar[r]^{y}&C\ar[d]^{\eta_{C}}\ar@{-->}[r]&\\
j^\ast (Q)\ar[r] & j^\ast j_!( B)\ar[r]^-{j^\ast j_!(y)}&j^\ast j_!( C)\ar@{-->}[r]&}$$
Since $j_!$ is fully faithful, the unit maps $\eta_{B}$ and $\eta_{C}$ are isomorphisms. Then so does $\alpha$. Then we have the isomorphisms $j_! (A)\cong j_!j^\ast (Q)\cong Q$. Set the isomorphism  $\alpha':j_! (A)\rightarrow Q$. Then we get an $\E_{\B}$-triangle
\begin{equation}\label{eq3-9}
j_! (A)\xrightarrow{a\alpha'} j_{!}(B)\xrightarrow{j_{!}(y)}  j_{!}(C)\dashrightarrow.
\end{equation}
Since $j_!$ is fully faithful, there exist two morphisms $a':A\rightarrow B$ and $b':A\rightarrow A$ such that $j_{!}(a')=a\alpha'$ and $j_{!}(b')=\alpha'^{-1}b$. Then $j_{!}(a')j_{!}(b')=a\alpha'\alpha'^{-1}b=ab=j_{!}(x)$. Thus, $j_{!}(x-a'b')=0$ and so, $x-a'b'=0$ since $j_!$ is fully faithful. It means that $x=a'b'$.\par
Applying the exact functor $j^{\ast}$ to $\E_{\B}$-triangle (\ref{eq3-9}), we have the following commutative diagram
$$\xymatrix{
A\ar[r]^{a'}\ar[d]^{\eta_{A}}&B\ar[d]^{\eta_{B}}\ar[r]^{y}&C\ar[d]^{\eta_{C}}&\\
j^\ast j_! (A)\ar[r]^{j_{!}(a')} & j^\ast j_!( B)\ar[r]^-{j^\ast j_!(y)}&j^\ast j_!( C)&}$$
Thus, $A\xrightarrow{a'} B\xrightarrow{y}  C\dashrightarrow$ is an $\E_{\C}$-triangle since $\eta: j^\ast j_!\rightarrow id$ is a natural equivalence. Then we have  the following commutative diagram.
$$\xymatrix{
A\ar[r]^{x}\ar[d]^{b'}&B\ar@{=}[d]\ar[r]^{y}&C\ar@{=}[d]\ar@{-->}[r]&\\
A\ar[r]^{a'} & B\ar[r]^-{y}& C\ar@{-->}[r]&}$$
Thus, we know that $b'$ is an isomorphism and so does $j_{!}(b')$. It implies that $b$ is an isomorphism since $j_{!}(b')=\alpha'^{-1}b$. Then we we have the following commutative diagram
$$\xymatrix{
j_! (A)\ar[r]^{j_{!}(x)}\ar[d]^{b}&j_!( B)\ar@{=}[d]\ar[r]^{ j_!(y)}&j_!( C)\ar@{=}[d]&\\
Q\ar[r]^{a} &  j_!( B)\ar[r]^-{ j_!(y)}& j_!( C)&}$$
Therefore, $j_{!}(A)\overset{j_{!}(x)}{\longrightarrow}j_{!}(B)\overset{j_{!}(y)}{\longrightarrow}j_{!}(C)\dashrightarrow$ is an $\E_{\B}$-triangle.
\end{proof}

At last, we will remark that if ($\mathcal{A}$, $\mathcal{B}$, $\mathcal{C}$) is  an $\s$-recollement, then $\C$ is  a localization of $\mathcal{B}$ in the sense of \cite[Theorem 3.5]{NOS}.

\begin{proposition}\label{prop3-5}
Let {\rm($\mathcal{A}$, $\mathcal{B}$, $\mathcal{C}$)} be  an $\s$-recollement of extriangulated categories. Then there is a unique equivalence functor $\widetilde{j^\ast}:\widetilde{\mathcal{B}}\rightarrow \mathcal{C}$,  where $\widetilde{\mathcal{B}}$, equipped
with an exact functor $Q:\B\rightarrow \widetilde{\B}$,  is the localization of $\mathcal{B}$ by {\rm$\mathcal{I}={j^{\ast}}^{-1}(\textrm{Iso}\C)$} such that $j^{\ast}=\widetilde{j^\ast}Q$.
\end{proposition}
\begin{proof}
Assume that ($\mathcal{A}$, $\mathcal{B}$, $\mathcal{C}$) is  an $\s$-recollement. Let $\mathcal{I}={j^{\ast}}^{-1}(\textrm{Iso}\C)$. By Corollary \cite[Corollary 3.8]{NOS}, there is a unique exact functor $\widetilde{j^\ast}:\widetilde{\mathcal{B}}\rightarrow \mathcal{C}$ such that $j^{\ast}=\widetilde{j^\ast}Q$, where $Q$ is a composition of canonical quotient functor $p:\B\rightarrow \B/[N_{\mathcal{I}}]$ and localization functor $\overline{Q}: \B/[N_{\mathcal{I}}]\rightarrow \widetilde{\mathcal{B}}$, see \cite{NOS} for the detail. It is clearly that $N_{\mathcal{I}}=\Ker j^{\ast}=\Im i_{\ast}$. Then, $Q(\Im i_{\ast} )=0$.\par
Next, we claim that $\widetilde{j^\ast}$ is an equivalence and its inverse is $\widetilde{j_{!}}=Q\circ j_{!}$. For any $X\in\B$, since $j_{!}$ is fully faithful, we have that
\begin{align*}
  \widetilde{j_{!}}\widetilde{j^\ast}(Q(X)) &=Q(j_{!}\widetilde{j^\ast}(Q(X)))  \\
   &=Q(j_{!}(\widetilde{j^\ast}Q(X)))\\
   &=Q(j_{!}(j^\ast(X)))\\
   &\cong Q(X).
\end{align*}
Moreover, from Lemma \ref{lem3-1} and Lemma \ref{lem3-2}, we have two  $\mathbb{E}_{\mathcal B}$-triangles:
\begin{align*}
  i_\ast A'\rightarrow j_! j^\ast X\rightarrow \textrm{CoCone}(\nu_X)\dashrightarrow\\
\textrm{CoCone}(\nu_X)\rightarrow X\rightarrow i_\ast i^\ast X\dashrightarrow.
\end{align*}
Applying the exact functor $Q$ to the above $\mathbb{E}_{\mathcal B}$-triangles, we have the following $\mathbb{E}_{\widetilde{\B}}$-triangles:
\begin{align*}
  Qi_\ast A'\rightarrow Qj_! j^\ast X\rightarrow Q(\textrm{CoCone}(\nu_X))\dashrightarrow\\
Q(\textrm{CoCone}(\nu_X))\rightarrow Q(X)\rightarrow Qi_\ast i^\ast X\dashrightarrow.
\end{align*}
Since $ Qi_\ast A'=0=Qi_\ast i^\ast X$, we have the isomorphism
$Qj_! j^\ast X\cong Q(X)$. Note that $j^{\ast}=\widetilde{j^\ast}Q$. It follows that $\widetilde{j_! }\widetilde{j^\ast}( Q(X))\cong Q(X)$. It completes the proof.
\end{proof}
\begin{corollary}\label{cor3-2}
 Let {\rm($\mathcal{A}$, $\mathcal{B}$, $\mathcal{C}$)} be  an $\s$-recollement of extriangulated categories. If $\mathcal{A}$, $\mathcal{B}$, and $\mathcal{C}$ are triangulated {\rm(}or, abelian {\rm)} categories, then there is an  equivalence between $\widetilde{\mathcal{B}}$ and $\B/[i_\ast(\mathcal{A})]$, where $\widetilde{\mathcal{B}}$ is the localization of $\mathcal{B}$ by {\rm$\mathcal{I}={j^{\ast}}^{-1}(\textrm{Iso}\C)$} $\widetilde{\mathcal{B}}$.
\end{corollary}
\begin{proof} By Corollary \ref{cor3-1} (or, Proposition \ref{prop3-4}), we know that $\C$ is equivalent to $\B/[i_\ast(\mathcal{A})]$. From Proposition \ref{prop3-5}, we get the desired equivalence.
\end{proof}

We end this section with the following result which describes  the localization category $\C$ of $\mathcal{B}$ by perpendicular subcategories.
\begin{lemma}\label{lemma4-2}
Let {\rm($F$, $G$)} be an adjoint pair between extriangulated categories $\A$ and $\B$. The following hold.
\begin{enumerate}
  \item If $G:\B\rightarrow\A$  is exact, then $F$ preserves projective objects.
  \item If $F:\A\rightarrow\B$  is exact, then $G$ preserves injective objects.
\end{enumerate}
\end{lemma}
\begin{proof} We just prove (1) and (2) can be obtained by dual arguments.

 (1) This proof is similar to \cite[Lemma 3.3 (3)]{WWZ}. For convenience, we give the details. \par
 Let $P$ be a projective object of $\A$. Assume that
 $$X\xrightarrow{a} Y\xrightarrow{b} Z\dashrightarrow$$
 is an arbitrary $\E_{\B}$-triangle. As $G:\B\rightarrow\A$  is exact, we have an  $\E_{\A}$-triangle
 $$G(X)\xrightarrow{G(a)} G(Y)\xrightarrow{G(b)} G(Z)\dashrightarrow.$$
 Applying the $\Hom_{\B}(F(P),-)$ to the above $\E_{\A}$-triangle, by adjointness of ($F$, $G$), we have the following commutative diagram
 $$\xymatrix{\Hom_{\B}(F(P),Y)\ar[rr]^{\Hom_{\B}(F(P),b)}\ar[d]_{\eta_{P,Y}}&& \Hom_{\B}(F(P),Z) \ar[d]^{\eta_{P,Z}} \\
 \Hom_{\A}(P,G(Y))\ar[rr]^{ \Hom_{\A}(P,G(b))}&&\Hom_{\A}(P,G(Z))  }$$
 where the vertical  arrows are isomorphisms.
Since $P$ is projective, $\Hom_{\A}(P,G(b))$ is epic. Note that $\eta_{P,Y}$ and $\eta_{P,Z}$ are isomorphisms. Then, $\Hom_{\B}(F(P),b)$ is epic. Therefore, $F(P)$ is a projective object of $\B$.
\end{proof}

\begin{proposition}\label{prop3-7}
Let {\rm($\mathcal{A}$, $\mathcal{B}$, $\mathcal{C}$)} be an $\s$-recollement of extriangulated categories. Assume that $\C$ has enough projective and injective objects. Set
{\rm\begin{align*}
\mathcal{X}=&\Big\{B\in\B\Big|
\begin{array}{c}
  ~\exists ~\text{a ~right~ $\s$-exact~sequence}~j_{!}(P_{1})\rightarrow j_{!}(P_{0})\xrightarrow{f} B~\text{in}~\B, \\
  ~{\rm where}~P_{i}\in\proj\C, ~\text{CoCone}(f)\in \Ker\Hom_{\B}(-,i_{\ast}(\A))
\end{array}\Big\} \\
\mathcal{Y}=&\Big\{B\in\B\Big|
\begin{array}{c}
  \exists ~\text{a ~left~ $\s$-exact~sequence}~B \xrightarrow{g} j_{\ast}(I_{0})\rightarrow j_{\ast}(I_{1}) ~\text{in}~\B, \\
  ~{\rm where}~I_{i}\in\inj\C, ~\text{Cone}(g)\in \Ker\Hom_{\B}(i_{\ast}(\A),-)
\end{array}\Big\}.
\end{align*}}
 Then the following hold.
\begin{enumerate}
  \item If $j_{!}$ is right $\s$-exact, then {\rm$\mathcal{X}\bigcap\Ker\Hom_{\B}(-,i_{\ast}(\A))={^{_{0}\perp_{1}}}i_{\ast}(\A)$} and there is an equivalence ${^{_{0}\perp_{1}}}i_{\ast}(\A)\rightarrow \C$, where {\rm${^{_{0}\perp_{1}}}i_{\ast}(\A)=\Ker\Hom_{\B}(-,i_{\ast}(\A))\bigcap\Ker\E_{\B}(-,i_{\ast}(\A))$}.
  \item If $j_{\ast}$ is left $\s$-exact, then {\rm$\mathcal{Y}\bigcap\Ker\Hom_{\B}(i_{\ast}(\A),-)=i_{\ast}(\A){^{_{0}\perp_{1}}}$} and there is an equivalence $i_{\ast}(\A){^{_{0}\perp_{1}}}\rightarrow \C$, where {\rm$i_{\ast}(\A){^{_{0}\perp_{1}}}=\Ker\Hom_{\B}(i_{\ast}(\A),-)\bigcap\Ker\E_{\B}(i_{\ast}(\A),-)$}.
\end{enumerate}
\end{proposition}
\begin{proof}
(1) Let $B\in\mathcal{X}\bigcap\Ker\Hom_{\B}(-,i_{\ast}(\A))$. Then there is an $\E_{\B}$-triangle
$$\text{CoCone}(f)\rightarrow j_{!}(P_{0})\xrightarrow{f} B\dashrightarrow$$
with $\text{CoCone}(f)\in \Ker\Hom_{\B}(-,i_{\ast}(\A))$. For any $A\in\A$, applying the functor $\Hom_{\B}(-,i_{\ast}(A))$ to the above $\E_{\B}$-triangle, we have the exact sequence
$$\rightarrow\Hom_{\B}(j_{!}(P_{0}),i_{\ast}(A))\rightarrow
\Hom_{\B}(\text{CoCone}(f),i_{\ast}(A))\rightarrow \E_{\B}(B,i_{\ast}(A))\rightarrow \E_{\B}(j_{!}(P_{0}),i_{\ast}(A))\rightarrow$$
By Lemma \ref{lemma4-2}, $j_{!}(P_{0})$ is projective. It yields that $\E_{\B}(j_{!}(P_{0}),i_{\ast}(A))=0$. Using the adjoint pair ($j_{!}$,$j^{\ast}$), we infer that $\Hom_{\B}(j_{!}(P_{0}),i_{\ast}(A))\cong\Hom_{\C}(P_{0},j^{\ast}i_{\ast}(A))=0$ since $j^{\ast}i_{\ast}(A)=0$. It implies that $\E_{\B}(B,i_{\ast}(A))\cong\Hom_{\B}(\text{CoCone}(f),i_{\ast}(A))=0$. Thus, $\mathcal{X}\bigcap\Ker\Hom_{\B}(-,i_{\ast}(\A))\subseteq{^{_{0}\perp_{1}}}i_{\ast}(\A)$.

Now, suppose $B\in {^{_{0}\perp_{1}}}i_{\ast}(\A)$. Then $j^{\ast}(B)\in\C$. Since $\C$ has enough projective  objects, there is a deflation $\alpha:P_{0}\rightarrow j^{\ast}(B)$, where $P_{0}\in\proj\C$. Since $j_{!}$ is right $\s$-exact, $j_{!}(\alpha):j_{!}(P_{0})\rightarrow j_{!}j^{\ast}(B)$ is a deflation. By (SR4) and Lemma \ref{lem3-2}, there are two $\E_{\B}$-triangle
\begin{align*}
 &i_{\ast}(A')\rightarrow j_! j^\ast (B)\xrightarrow{g} \textrm{CoCone}(\nu_B)\dashrightarrow\\
&\textrm{CoCone}(\nu_B)\xrightarrow{f} B\xrightarrow{\nu_B} i_\ast i^\ast (B)\dashrightarrow
\end{align*}
where $fg=\upsilon_B$ is a counit map. Note that $\Hom_{\B}(B,i_{\ast}(\A))=0$. We infer that $\nu_B=0$. Thus, by (ET3)$^{\textrm{op}}$, we have the following commutative diagram
$$\xymatrix{ B\ar[r]^{1}\ar@{-->}[d]^{\beta}&B\ar[d]^{1}\ar[r]&0\ar@{-->}[r]\ar[d]^{0}&\\
\textrm{CoCone}(\nu_B)\ar[r]^-{f}& B\ar[r]^{0} &i_\ast i^\ast (B)\ar@{-->}[r]&               }$$
Thus, $f\beta=1$ and hence $f$ is a deflation since $f$ is a retraction. It yields that $\upsilon_B=fg $ is a deflation by \cite[Remark 2.16]{NP}. Thus, $\upsilon_B j_{!}(\alpha):j_{!}(P_{0})\rightarrow B$ is a deflation. Then there is an $\E_{\B}$-triangle
$$K_{0}\xrightarrow{\lambda} j_{!}(P_{0})\xrightarrow{\upsilon_B j_{!}(\alpha)} B\dashrightarrow.$$
where $K_{0}=\textrm{CoCone}(\upsilon_B j_{!}(\alpha))$.
Applying the functor $\Hom_{\B}(-,i_{\ast}(A))$ to the above $\E_{\B}$-triangle, we have the exact sequence
$$\rightarrow\Hom_{\B}(j_{!}(P_{0}),i_{\ast}(A))\rightarrow
\Hom_{\B}(K_{0},i_{\ast}(A))\rightarrow \E_{\B}(B,i_{\ast}(A))\rightarrow \E_{\B}(j_{!}(P_{0}),i_{\ast}(A))\rightarrow$$
It is easy to see that $\Hom_{\B}(j_{!}(P_{0}),i_{\ast}(A))=\E_{\B}(B,i_{\ast}(A))=0$. Thus, $K_{0}\in\Ker\Hom_{\B}(-,i_{\ast}(A))$. Repeating the above arguments, there is a deflation $\gamma:j_{!}(P_{1})\rightarrow K_{0}$. Thus, there is a right $\s$-exact sequence $j_{!}(P_{1})\xrightarrow{\lambda\gamma}j_{!}(P_{0})\xrightarrow{\upsilon_B j_{!}(\alpha)}B\dashrightarrow  $ with $\textrm{CoCone}(\upsilon_B j_{!}(\alpha))\in\Ker\Hom_{\B}(-,i_{\ast}(A))$. Therefore, $B\in\mathcal{X}\bigcap\Ker\Hom_{\B}(-,i_{\ast}(\A))$ and so, ${^{_{0}\perp_{1}}}i_{\ast}(\A)\subseteq \mathcal{X}\bigcap\Ker\Hom_{\B}(-,i_{\ast}(\A))$.\par
Next, we show that $j^{\ast}$ induces the equivalence ${^{_{0}\perp_{1}}}i_{\ast}(\A)\rightarrow \C$.\par
Assume that $B\in{^{_{0}\perp_{1}}}i_{\ast}(\A)$. Then, from the above arguments, there is an $\E_{\B}$-triangle
$$i_{\ast}(A')\rightarrow j_! j^\ast (B)\xrightarrow{\upsilon_B} B\dashrightarrow$$
Since $\E_{\B}(B,i_{\ast}(\A))=0$, the above $\E_{\B}$-triangle is split and hence, $j_! j^\ast (B)\cong B\oplus i_{\ast}(A')$. Moreover, applying $\Hom_{\B}(-,i_{\ast}(A'))$ to this split $\E_{\B}$-triangle, we have the following exact sequence
$$\rightarrow\Hom_{\B}(j_! j^\ast (B),i_{\ast}(A'))\rightarrow
\Hom_{\B}(i_{\ast}(A'),i_{\ast}(A'))\rightarrow \E_{\B}(B,i_{\ast}(A'))\rightarrow$$
Note that $\Hom_{\B}(j_! j^\ast (B),i_{\ast}(A'))\cong\Hom_{\C}( j^\ast (B),j^{\ast}i_{\ast}(A'))=0 $ and $\E_{\B}(B,i_{\ast}(A'))=0$ since $j^{\ast}i_{\ast}=0$ and $\E_{\B}(B,i_{\ast}(\A))=0$. Thus, $\Hom_{\B}(i_{\ast}(A'),i_{\ast}(A'))=0$ and so $i_{\ast}(A')=0$. Therefore, $j_! j^\ast (B)\cong B$. It implies that $j^\ast:{^{_{0}\perp_{1}}}i_{\ast}(\A)\rightarrow \C$ is an equivalence.\par
(2) can be proved by similar arguments.
\end{proof}
\begin{corollary}{\rm(\cite[Proposition 3.2]{P})}
Let {\rm($\mathcal{A}$, $\mathcal{B}$, $\mathcal{C}$)} be a recollement of abelian categories. Assume that $\C$ has enough projective and injective objects. Set
{\rm\begin{align*}
\mathcal{X'}=&\{B\in\B|~\exists ~\text{a ~right~ exact~sequence}~j_{!}(P_{1})\rightarrow j_{!}(P_{0})\xrightarrow{f} B~\text{in}~\B,~{\rm where}~P_{i}\in\proj\C\} \\
\mathcal{Y'}=&\{B\in\B|~\exists ~\text{a ~left~exact~sequence}~B \xrightarrow{g} j_{\ast}(I_{0})\rightarrow j_{\ast}(I_{1}) ~\text{in}~\B,
  ~{\rm where}~I_{i}\in\inj\C \}\\
{^{_{0}\perp_{1}}}i_{\ast}(\A)&=\Ker\Hom_{\B}(-,i_{\ast}(\A))\bigcap\Ker\E_{\B}(-,i_{\ast}(\A))
\\i_{\ast}(\A){^{_{0}\perp_{1}}}&=\Ker\Hom_{\B}(i_{\ast}(\A),-)\bigcap\Ker\E_{\B}(i_{\ast}(\A),-).
\end{align*}}
 Then the following hold.
\begin{enumerate}
  \item {\rm$\mathcal{X'}={^{_{0}\perp_{1}}}i_{\ast}(\A)$} and there is an equivalence $\mathcal{X'}\rightarrow \C$.
  \item  {\rm$\mathcal{Y'}=i_{\ast}(\A){^{_{0}\perp_{1}}}$} and there is an equivalence $\mathcal{Y'}\rightarrow \C$.
\end{enumerate}
\end{corollary}
\begin{proof}It is enough to show that $\mathcal{X}\bigcap\Ker\Hom_{\B}(-,i_{\ast}(\A))=\mathcal{X'}$. Clearly, one can see that  $\mathcal{X}\bigcap\Ker\Hom_{\B}(-,i_{\ast}(\A))\subseteq\mathcal{X'}$. Next, for any $B\in\mathcal{X'}$, there are two short exact sequences
\begin{align*}
0\rightarrow K_{1}\rightarrow j_{!}(P_{1})\rightarrow K_{0}\rightarrow 0 \\
0\rightarrow K_{0}\rightarrow j_{!}(P_{0})\xrightarrow{f} B\rightarrow 0
\end{align*}
where $K_{0}=\Ker f$. For any $A\in\A$, applying the functor $\Hom_{\B}(-,i_{\ast}(\A))$ to them, we have the following exact sequences
\begin{align*}
&0\rightarrow \Hom_{\B}(K_{0},i_{\ast}(\A))\rightarrow\Hom_{\B}(j_{!}(P_{1}),i_{\ast}(\A)) \rightarrow \Hom_{\B}(K_{1},i_{\ast}(\A))\rightarrow \\
&0\rightarrow \Hom_{\B}(B,i_{\ast}(\A))\rightarrow \Hom_{\B}(j_{!}(P_{0}),i_{\ast}(\A))\rightarrow \Hom_{\B}(K_{0},i_{\ast}(\A))\rightarrow
\end{align*}
Using the adjoint pair ($j_{!}$, $j^{\ast}$), we have that $\Hom_{\B}(j_{!}(P_{1}),i_{\ast}(\A))\cong\Hom_{\C}(P_{1},j^{\ast}i_{\ast}(\A))=0$ since $j^{\ast}i_{\ast}=0$. Thus, $\Hom_{\B}(K_{0},i_{\ast}(\A))=0$. Similarly, $\Hom_{\B}(B,i_{\ast}(\A))=0$. This completes the proof. By Proposition \ref{prop3-7} (1), we show the statement (1).

 Finally, similar arguments prove the statement (2).
\end{proof}
\section{Applications}

\begin{definition}\cite{NOS}
Let $(\mathcal{A},\mathbb{E},\mathfrak{s})$  be an extriangulated category
\begin{enumerate}
  \item A full additive subcategory $\mathcal{N}\subseteq \A$ is called a \textit{thick} subcategory if it satisfies the following conditions.
\begin{enumerate}
  \item $\mathcal{N}\subseteq \A$ is closed by isomorphisms and direct summands.
  \item $\mathcal{N}$ satisfies 2-out-of-3 for conflations. Namely, if any two of objects $A$, $B$, $C$ in an conflation $A\xrightarrow{x} B \xrightarrow{y} C$ belong to $\mathcal{N}$, then so does the third.
\end{enumerate}
  \item A thick subcategory $\mathcal{N}\subseteq \A$ is called \textrm{biresolving}, if for any $A\in\A$
there exists an inflation $C \rightarrow N$ and a deflation $N' \rightarrow C$ for some $N$, $N'\in \mathcal{N}$.
\end{enumerate}
\end{definition}
\begin{definition}\cite{NOS}
Let $(\mathcal{A},\mathbb{E},\mathfrak{s})$  be an extriangulated category. For a thick subcategory $\mathcal{N}\subseteq \A$, we associate the following classes
of morphisms.
\begin{align*}
  \mathcal{R} &= \{f\in\M|~f~\text{is~an~inflation~with}~ \textrm{Cone(f)}\in\mathcal{N}\} \\
  \mathcal{L} &=\{f\in\M|~f~\text{is~a~deflation~with}~ \textrm{CoCone(f)}\in\mathcal{N}\}
\end{align*}
Define $\I_{\mathcal{N}}\subseteq \M$ to be the smallest subset closed by compositions containing
both $\mathcal{R}$ and $\mathcal{L}$. It is obvious that $\I_{\mathcal{N}}$ satisfies condition (M0) in Section \ref{sect2.3}.
\end{definition}
\begin{remark}\cite{NOS}
$\I_{\mathcal{N}}$ coincides with the set of all finite compositions of morphisms in $\mathcal{R}$ and $\mathcal{L}$. Moreover, $\mathcal{N}_{\I_{\mathcal{N}}}=\mathcal{N}$, see \cite[Lemma 4.5]{NOS}.
\end{remark}
\begin{lemma}\label{lemm4-1}
Let $\mathcal{A}$ and $\B$  be  extriangulated categories. Assume that $\mathcal{N}_{1}$ and $\mathcal{N}_{2}$ are thick subcategories of $\mathcal{A}$ and $\B$, respectively. If the exact functor $F$ satisfies $F(\mathcal{N}_{1})\subseteq \mathcal{N}_{2}$, then $F(\I_{\mathcal{N}_{1}})\subseteq \I_{\mathcal{N}_{2}}$.
\end{lemma}
\begin{proof}Let $f\in\Hom_{\A}(X,Y) $. If $f\in\mathcal{R}_{\A}$, then  there is a $\E_{A}$-triangle $X\xrightarrow{f}Y\rightarrow \textrm{Cone}(f)\dashrightarrow$. Applying the exact functor $F$ to this triangle, we have the $\E_{B}$-triangle $F(X)\xrightarrow{F(f)}F(Y)\rightarrow F(\textrm{Cone}(f))\dashrightarrow$. Thus, $\textrm{Cone}(F(f))\cong F(\textrm{Cone}(f))\in\mathcal{N}_{2}$ since $F(\mathcal{N}_{1})\subseteq \mathcal{N}_{2}$ and $\textrm{Cone}(f)\in \mathcal{N}_{1}$. Thus, $F(\mathcal{R}_{\A})\subseteq \mathcal{R}_{\B}$. Dually, we can prove that $F(\mathcal{L}_{\A})\subseteq \mathcal{L}_{\B}$. Note that $\I_{\mathcal{N}_{1}}$ coincides with the set of all finite compositions of morphisms in $\mathcal{R}_{\A}$ and $\mathcal{L}_{\A}$. Hence, for any $h\in\I_{\mathcal{N}_{1}}$, $F(h)\in\I_{\mathcal{N}_{2}}$.
\end{proof}
\begin{proposition}{\rm\cite{NOS}}\label{prop4-0}
Let $\mathcal{A}$  be an extriangulated category satisfying {\rm(WIC)} condition, or $\mathcal{A}$  be an exact category. If  $\mathcal{N}$ is a biresolving thick subcategory of $\A$, then the following hold.
\begin{enumerate}
  \item $\I_{\mathcal{N}}=\mathcal{R}\circ\mathcal{L}$.
  \item $p(\I_{\mathcal{N}})=\overline{\I_{\mathcal{N}}}$.
  \item $\I_{\mathcal{N}}$ satisfies {\rm(MR1), . . . , (MR4)}.
  \item The localization of $\A$ by $\I_{\mathcal{N}}$ corresponds to a
triangulated category.
\end{enumerate}
\end{proposition}
\begin{definition}\label{def4-3}
 (1)(\cite{BBD}) A recollement of triangulated categories is a diagram of triangulated categories and triangle functors
$$\xymatrix{\mathcal{A}\ar[rr]|{i_{*}}&&\ar@/_1pc/[ll]|{i^{*}}\ar@/^1pc/[ll]|{i^{!}}\mathcal{B}
\ar[rr]|{j^{\ast}}&&\ar@/_1pc/[ll]|{j_{!}}\ar@/^1pc/[ll]|{j_{\ast}}\mathcal{C}},$$
satisfying \\
\begin{enumerate}
  \item [(R1)]$(i^{*},i_{\ast})$, $(i_{\ast},i^{!})$, and $(j_{!},j^{\ast})$, $(j^{\ast},j_{*})$ are adjoint pairs;
  \item [(R2)]$i_{\ast}$, $j_{*}$, $j_{!}$ are fully faithful functors;
  \item [(R3)]$j^{\ast}\circ i_{\ast}=0$;
  \item [(R4)]for each $X\in \mathcal{B}$ there are triangles
$$j_{!}j^{\ast}(X)\rightarrow X \rightarrow i_{\ast}i^{*}(X)\rightarrow i_{\ast}i^{!}(X)[1]$$
$$i_{\ast}i^{!}(X)\rightarrow X \rightarrow j_{*}j^{\ast}(X)\rightarrow j_{*}j^{\ast}(X)[1],$$
\end{enumerate}
where the arrows to and from $C$ are the counit and the unit, respectively.\\
(2) (\cite{P0,K,BGS}) A lower recollement of $\mathcal{C}'$ and $\mathcal{C}''$ is a diagram of triangle functors
$$\xymatrix{\mathcal{A}\ar[rr]|{i_{*}}&&\ar@/^1pc/[ll]|{i^{!}}\mathcal{B}
\ar[rr]|{j^{\ast}}&&\ar@/^1pc/[ll]|{j_{\ast}}\mathcal{C}}.$$
such that the conditions involved $i^{!}$, $i_{\ast}$, $j_{\ast}$, $j^{\ast}$ in (1) are satisfied. Dually, one can define the upper recollement.
\end{definition}

In what follows,  add$\mathcal{X}\subseteq \A$ denotes the smallest additive full subcategory containing $\mathcal{X}$, closed by isomorphisms and  direct summands. All extriangulated categories containing in  $\s$-recollements are assumed to satisfy the (WIC) condition.
\begin{proposition}\label{prop4-1}
Let {\rm($\mathcal{A}$, $\mathcal{B}$, $\mathcal{C}$)} be an $\s$-recollement of extriangulated categories. Assume that  $\mathcal{N}$ is a biresolving thick subcategory of $\B$ such that $i_{\ast}i^{!}(\mathcal{N})\subseteq \mathcal{N}$. Set {\rm$\mathcal{N}_{1}=\add i^{!}\mathcal{N}$} and {\rm$\mathcal{N}_{2}=\add j^{\ast}\mathcal{N}$}. If $i^{!}$ is exact and $j_{\ast}$ is left $\s$-exact, then there is a lower recollement of triangulated categories.
$$\xymatrix{\widetilde{\mathcal{A}}\ar[rr]|{\widetilde{i_{*}}}&&\ar@/^1pc/[ll]|{\widetilde{i^{!}}}\widetilde{\mathcal{B}}
\ar[rr]|{\widetilde{j^{\ast}}}&&\ar@/^1pc/[ll]|{\widetilde{j_{\ast}}}\widetilde{\mathcal{C}}}.$$
where $\widetilde{\mathcal{A}}$, $\widetilde{\mathcal{B}}$ and $\widetilde{\mathcal{C}}$ are the localizations by {\rm$\I_{\mathcal{N}_{1}}$}, $\I_{\mathcal{N}}$ and {\rm$\I_{\mathcal{N}_{2}}$}.
\end{proposition}

\begin{proof}
\par
First, we show that $\mathcal{N}_{1}$ and $\mathcal{N}_{2}$ are biresolving thick subcategory of $\A$ and $\C$, respectively.\par
For any conflation $X\rightarrow Y\rightarrow Z$, since $i_{\ast}$ is exact, we have a conflation $i_{\ast}(X)\rightarrow i_{\ast}(Y)\rightarrow i_{\ast}(Z)$ in $\B$.
Since $i_{\ast}i^{!}(\mathcal{N})\subseteq \mathcal{N}$ and $i_{\ast}$ is fully faithful, by the 2-out-of-3 property of $\mathcal{N}$, we know that $\mathcal{N}_{1}$ satisfies the 2-out-of-3 property.\par
For any $X\in\A$, $i_{\ast}(X)\in\B$. Then there is an  inflation $i_{\ast}(X)\rightarrow N$ and a deflation $N'\rightarrow i_{\ast}(X)$.  Since $i^{!}$ is exact, $i^{!}i_{\ast}(X)\rightarrow i^{!}(N)$ is an  inflation and $i^{!}(N')\rightarrow i^{!}i_{\ast}(X)$ is a deflation. Since $i_{\ast}$ is fully faithful, $i^{!}i_{\ast}(X)\cong X$. Thus, $\mathcal{N}_{1}$ is a biresolving thick subcategory of $\A$.\par
Since $i^{!}$ is exact and $j_{\ast}$ is left $\s$-exact, we know that $j_{\ast}$ is exact by Proposition \ref{prop3-6}. Moreover, by Proposition \ref{prop3-3}, there is an  $\mathbb{E}_\mathcal{B}$-triangle
\begin{equation}\label{eq4-1}
  i_\ast i^! (X)\rightarrow X\rightarrow j_\ast j^\ast (X)\dashrightarrow
\end{equation}
for any $X\in\B$. Then, for any $X\in\mathcal{N}$, by the 2-out-of-3 property, $j_\ast j^\ast (X)\in \mathcal{N}$. Hence, $j_\ast j^\ast (\mathcal{N})\subseteq \mathcal{N}$. By similar arguments of $\mathcal{N}_{1}$, we know that $\mathcal{N}_{2}$ is a biresolving thick subcategory of $\B$.\par
Now, we know that $i^{!}\mathcal{N}\subseteq \mathcal{N}_{1}$, $j^\ast (\mathcal{N})\subseteq \mathcal{N}_{2}$, $i_{\ast}(\mathcal{N}_{1})\subseteq \mathcal{N}$ and $j_\ast(\mathcal{N}_{2})\subseteq \mathcal{N}$. Moreover, these four functors are all exact. By Lemma \ref{lemm4-1}, Theorem \ref{prop2-1}, and Proposition \ref{prop4-0}, there is a diagram of exact  functors
\begin{equation}\label{eq4-2}
  \text{$\xymatrix{\widetilde{\mathcal{A}}\ar[rr]|{\widetilde{i_{*}}}&&\ar@/^1pc/[ll]|{\widetilde{i^{!}}}\widetilde{\mathcal{B}}
\ar[rr]|{\widetilde{j^{\ast}}}&&\ar@/^1pc/[ll]|{\widetilde{j_{\ast}}}\widetilde{\mathcal{C}}}.$}
\end{equation}
where $\widetilde{\mathcal{A}}$, $\widetilde{\mathcal{B}}$ and $\widetilde{\mathcal{C}}$ are the localizations by {\rm$\I_{\mathcal{N}_{1}}$}, $\I_{\mathcal{N}}$ and {\rm$\I_{\mathcal{N}_{2}}$}, ($\widetilde{i_{*}}$, $\widetilde{i^{!}}$) and ($\widetilde{j^{\ast}}$, $\widetilde{j_{\ast}}$) are adjoint pairs. By Corollary \ref{cor2-2}, we know that $\widetilde{i_{*}}$ and $\widetilde{j_{\ast}}$ are fully faithful. By Corollary \ref{cor2-1}, $\widetilde{i_{*}}\widetilde{j^\ast}(Q_{\A}(X))=Q_{\A}i_{*}j^\ast(X)=Q_{\A}(0)=0$,  $Q_{\A}i_\ast i^! (X)=\widetilde{i_\ast} \widetilde{i^!} (Q_{\A}(X))$ and $Q_{\A}j_\ast j^\ast (X)=\widetilde{j_\ast} \widetilde{j^\ast}(Q_{\A}(X))$, for any $X\in\B$. Then, applying the exact functor $Q_{A}$ to the $\mathbb{E}_\mathcal{B}$-triangle (\ref{eq4-1}), we have the following $\widetilde{\E}_{\widetilde{\B}}$-triangle
$$\widetilde{i_\ast} \widetilde{i^!}(Q_{\A}(X))\rightarrow Q_{\A}(X)\rightarrow \widetilde{j_\ast} \widetilde{j^\ast} (Q_{\A}(X))\dashrightarrow,$$
which is a triangle of triangulated category $\widetilde{\B}$.\par
Therefore, the diagram (\ref{eq4-2}) is a lower recollement of triangulated categories.
\end{proof}
Dually, we have the following result.
\begin{proposition}
Let {\rm($\mathcal{A}$, $\mathcal{B}$, $\mathcal{C}$)} be an $\s$-recollement of extriangulated categories. Assume that  $\mathcal{N}$ is a biresolving thick subcategory of $\B$ such that $i_{\ast}i^{\ast}(\mathcal{N})\subseteq \mathcal{N}$. Set {\rm$\mathcal{N}_{1}=\add i^{\ast}\mathcal{N}$} and {\rm$\mathcal{N}_{2}=\add j^{\ast}\mathcal{N}$}. If $i^{\ast}$ is exact and $j_{!}$ is right $\s$-exact, then there is an upper recollement of triangulated categories.
$$\xymatrix{\widetilde{\mathcal{A}}\ar[rr]|{\widetilde{i_{*}}}&&\ar@/_1pc/[ll]|{\widetilde{i^{*}}}\widetilde{\mathcal{B}}
\ar[rr]|{\widetilde{j^{\ast}}}&&\ar@/_1pc/[ll]|{\widetilde{j_{!}}}\widetilde{\mathcal{C}}},$$
where $\widetilde{\mathcal{A}}$, $\widetilde{\mathcal{B}}$ and $\widetilde{\mathcal{C}}$ are the localizations by {\rm$\I_{\mathcal{N}_{1}}$}, $\I_{\mathcal{N}}$ and {\rm$\I_{\mathcal{N}_{2}}$}.
\end{proposition}
\begin{corollary}\label{cor4-1}
Let {\rm($\mathcal{A}$, $\mathcal{B}$, $\mathcal{C}$)} be an $\s$-recollement of extriangulated categories. Assume that  $\mathcal{N}$ is a biresolving thick subcategory of $\B$ such that $i_{\ast}i^{!}(\mathcal{N})\subseteq \mathcal{N}$ and $j_{!}j^{\ast}(\mathcal{N})\subseteq\mathcal{N}$. Set {\rm$\mathcal{N}_{1}=\add i^{!}\mathcal{N}$} and {\rm$\mathcal{N}_{2}=\add j^{\ast}\mathcal{N}$}. If both $i^{!}$ and $j_{!}$ are exact, $j_{\ast}$ is left $\s$-exact, then there is a recollement of triangulated categories.
\begin{equation}\label{dia4-1}
\xymatrix{\widetilde{\mathcal{A}}\ar[rr]|{\widetilde{i_{*}}}&&\ar@/_1pc/[ll]|{\widetilde{i^{*}}}\ar@/^1pc/[ll]|{\widetilde{i^{!}}}\widetilde{\mathcal{B}}
\ar[rr]|{\widetilde{j^{\ast}}}&&\ar@/_1pc/[ll]|{\widetilde{j_{!}}}\ar@/^1pc/[ll]|{\widetilde{j_{\ast}}}\widetilde{\mathcal{C}}},
\end{equation}
where $\widetilde{\mathcal{A}}$, $\widetilde{\mathcal{B}}$ and $\widetilde{\mathcal{C}}$ are the localizations by {\rm$\I_{\mathcal{N}_{1}}$}, $\I_{\mathcal{N}}$ and {\rm$\I_{\mathcal{N}_{2}}$}.
\end{corollary}
\begin{proof}
By Proposition \ref{prop4-1}, there is a lower recollement of triangulated categories.
$$\xymatrix{\widetilde{\mathcal{A}}\ar[rr]|{\widetilde{i_{*}}}&&\ar@/^1pc/[ll]|{\widetilde{i^{!}}}\widetilde{\mathcal{B}}
\ar[rr]|{\widetilde{j^{\ast}}}&&\ar@/^1pc/[ll]|{\widetilde{j_{\ast}}}\widetilde{\mathcal{C}}}.$$
where $\widetilde{\mathcal{A}}$, $\widetilde{\mathcal{B}}$ and $\widetilde{\mathcal{C}}$ are the localizations by {\rm$\I_{\mathcal{N}_{1}}$}, $\I_{\mathcal{N}}$ and {\rm$\I_{\mathcal{N}_{2}}$}. Then $\widetilde{\mathcal{A}}\cong \Ker j^{\ast}$.\par
Since $j_{!}j^{\ast}(\mathcal{N})\subseteq\mathcal{N}$ and $j_{!}$ is exact, by Theorem \ref{prop2-1} and Corollary \ref{cor2-2}, the exact functor $\widetilde{j^{\ast}}$ admits a fully faithful left adjoint functor $\widetilde{j_{!}}$. By \cite[Corollary 2.6]{YG}, the embedding  functor $\widetilde{i_{\ast}}$ admits a  left adjoint functor $\widetilde{i^{\ast}}$ such that $$\xymatrix{\widetilde{\mathcal{A}}\ar[rr]|{\widetilde{i_{*}}}&&\ar@/_1pc/[ll]|{\widetilde{i^{*}}}\widetilde{\mathcal{B}}
\ar[rr]|{\widetilde{j^{\ast}}}&&\ar@/_1pc/[ll]|{\widetilde{j_{!}}}\widetilde{\mathcal{C}}},$$
is an upper recollement. Then, we have the reollement (\ref{dia4-1}).
\end{proof}
\begin{corollary}
Let {\rm($\mathcal{A}$, $\mathcal{B}$, $\mathcal{C}$)} be an $\s$-recollement of extriangulated categories. Assume that  $\mathcal{N}$ is a biresolving thick subcategory of $\B$ such that $i_{\ast}i^{\ast}(\mathcal{N})\subseteq \mathcal{N}$ and $j_{\ast}j^{\ast}(\mathcal{N})\subseteq\mathcal{N}$. Set {\rm$\mathcal{N}_{1}=\add i^{!}\mathcal{N}$} and {\rm$\mathcal{N}_{2}=\add j^{\ast}\mathcal{N}$}. If both $i^{\ast}$ and $j_{\ast}$ are exact, $j_{!}$ is right $\s$-exact, then there is a recollement of triangulated categories.
$$\xymatrix{\widetilde{\mathcal{A}}\ar[rr]|{\widetilde{i_{*}}}&&\ar@/_1pc/[ll]|{\widetilde{i^{*}}}\ar@/^1pc/[ll]|{\widetilde{i^{!}}}\widetilde{\mathcal{B}}
\ar[rr]|{\widetilde{j^{\ast}}}&&\ar@/_1pc/[ll]|{\widetilde{j_{!}}}\ar@/^1pc/[ll]|{\widetilde{j_{\ast}}}\widetilde{\mathcal{C}}},$$
where $\widetilde{\mathcal{A}}$, $\widetilde{\mathcal{B}}$ and $\widetilde{\mathcal{C}}$ are the localizations by {\rm$\I_{\mathcal{N}_{1}}$}, $\I_{\mathcal{N}}$ and {\rm$\I_{\mathcal{N}_{2}}$}.
\end{corollary}
\begin{example}
Let $A$ be a finite dimension algebra over field $k$ and $T_{2}(A)$ be the upper triangular matrix algebra $\begin{bmatrix}\begin{smallmatrix} A & A\\0&A\end{smallmatrix}\end{bmatrix}
$. It is well-known that each left $\Lambda$-module  can be uniquely described as $\begin{bmatrix}\begin{smallmatrix} X \\Y\end{smallmatrix}\end{bmatrix}_{f}$, where $f:Y\rightarrow X$ is a left $\Lambda$-module homomorphism. \par
By \cite[Example 2.12]{P}, we have a recollement of module categories

\begin{equation}\label{eqEr1}
\xymatrix{\textrm{Mod} A\ar[rr]|{i_{*}}&&\ar@/_1pc/[ll]|{i^{*}}\ar@/^1pc/[ll]|{i^{!}}\textrm{Mod} T_{2}(A)
\ar[rr]|{j^{\ast}}&&\ar@/_1pc/[ll]|{j_{!}}\ar@/^1pc/[ll]|{j_{\ast}}\textrm{Mod} A.}
\end{equation}
where
\begin{align*}
i^{\ast}\big(\begin{bmatrix}\begin{smallmatrix} X \\Y\end{smallmatrix}\end{bmatrix}_{f}\big)&=\textrm{Coker} f & i_{\ast}(X)&=\begin{bmatrix}\begin{smallmatrix} X \\\textrm{O}\end{smallmatrix}\end{bmatrix}_{0}& i^{!}\big(\begin{bmatrix}\begin{smallmatrix} X \\Y\end{smallmatrix}\end{bmatrix}_{f}\big)&=X\\
j_!(Y)& =\begin{bmatrix}\begin{smallmatrix} Y \\Y\end{smallmatrix}\end{bmatrix}_{1}
& j^{\ast}\big(\begin{bmatrix}\begin{smallmatrix} X \\Y\end{smallmatrix}\end{bmatrix}_{f}\big)&=Y
& j_\ast(Y)&=\begin{bmatrix}\begin{smallmatrix}  \textrm{O} \\Y\end{smallmatrix}\end{bmatrix}_{0}.
\end{align*}
Here, all functors but $i^{\ast}$ are exact.\par
Now, we lift the recollement  (\ref{eqEr1}) of module categories to that of categories of chain complexes. Let $A$ and $\B$ be abelian categories. We denote by $\textrm{Ch}(\A)$ the category of chain
complexes over $\A$. Let $\textrm{Ch}_{ac}(\A)$ be the subcategory of acyclic complexes. By \cite[Example 2]{Rump}, we know that $\textrm{Ch}_{ac}(\A)$ is a biresolving subcategory of $\textrm{Ch}_{ac}(\A)$ and the localization $\widetilde{\textrm{Ch}(\A)}$ by $\I_{\textrm{Ch}_{ac}(\A)}$ is just the derived category $D(\A)$. Let $F:\A\rightarrow \B$ be a additive functor. For any $X^{\bullet}=(X^{n},d^{n})\in\textrm{Ch}(\A)$, define $F^{C}(X^{\bullet})$ (or written as $F_{C}(X^{\bullet})$) to be the complex of which the component and differentation of degree $n$ are $F(X^{n})$ and $F(d^{n})$ respectively, for any $n\in\mathbb{Z}$.  Then, by \cite[Theorem 5.3]{Bautista}, we have a new recollement of abelian categories
\begin{equation}\label{eqEr2}
\xymatrix{\textrm{Ch} (A)\ar[rr]|{i_{*}^{C}}&&\ar@/_1pc/[ll]|{i^{*}_{C}}\ar@/^1pc/[ll]|{i^{!}_{C}}\textrm{Ch} (T_{2}(A))
\ar[rr]|{j^{\ast}_{C}}&&\ar@/_1pc/[ll]|{j_{!}^{C}}\ar@/^1pc/[ll]|{j_{\ast}^{C}}\textrm{Ch} (A).}
\end{equation}
It is easy to check that all functors but $i^{\ast}_{C}$ are exact since all functors contained in the recollement  (\ref{eqEr1}) but $i^{\ast}$ are exact.  Note that $\mathcal{N}_{1}=j^{\ast}_{C}(\textrm{Ch}_{ac}(T_{2}(A)))=\textrm{Ch}_{ac}(A)$ and $\mathcal{N}_{2}=i^{!}(\textrm{Ch}_{ac}(T_{2}(A)))=\textrm{Ch}_{ac}(A)$. Since $j_{!}^{C}$  is exact, $j_{!}^{C}(\textrm{Ch}_{ac}(A))\subseteq \textrm{Ch}_{ac}(T_{2}(A))$. By Corollary \ref{cor4-1}, we get the following recollement of derived categories
 \begin{equation*}
\xymatrix{\textrm{D} (\textrm{Mod}A)\ar[rr]|{i_{*}}&&\ar@/_1pc/[ll]|{i^{*}}\ar@/^1pc/[ll]|{i^{!}}\textrm{D} (\textrm{Mod}T_{2}(A))
\ar[rr]|{j^{\ast}}&&\ar@/_1pc/[ll]|{j_{!}}\ar@/^1pc/[ll]|{j_{\ast}}\textrm{D} (\textrm{Mod}A).}
\end{equation*}

\end{example}


\vspace{1cm}

\begin{thebibliography}{100}
\bibitem{LAKL}L. Angeleri-H\"{u}gel, S. K\"{o}nig and Q. H. Liu. Recollements and tilting objects. J. Pure Appl. Algebra, 2011, 215:420-438.

\bibitem{BBD}  A. Beilinson, J. Bernstein, P. Deligne, Faisceaux pervers. Ast\'{e}risque, 1982, 100: 5-171.

\bibitem{BGS} A. Beilinson, V. Ginsburg, V. Schechtman, Koszul duality. J. Geom. Phys., 1998, 5(3): 317-350.

\bibitem{BR} A. Beligiannis, I. Reiten. Homological and homotopical aspects of torsion theories. Mem. Amer. Math. Soc., 2007, 188(883): viii+207.

\bibitem{B-TS} R. Bennett-Tennenhaus, A. Shah. Transport of structure in higher homological algebra. J. Algebra, 2021, 574: 514-549.

\bibitem{Bautista} R. Bautista and S. Liu, Covering theory for linear categories with application to derived categories, J. Algebra, 2014, 406:173-225.

\bibitem{BaP}S. Bazzoni and A. Pavarin. Recollements from partial tilting complexes. J. Algebra, 2013, 388:338-363.

\bibitem{CPS} E. Cline, B. Parshall, L. Scott, Finite-dimensional algebras and highest weight categories, J.Reine Angew. Math. 1988, 391: 85-99.
    
\bibitem{ChenJM} J. M. Chen. Cotorsion pairs in a recollement of triangulated categories. Comm. Algebra, 2013, 41:2903-2915.
    
\bibitem{CX}	H. X. Chen, C. C. Xi. Good tilting modules and recollements of derived module categories.Proc. London Math. Soc., 2012, 104(3): 959-996.

\bibitem{CX01}	H. X. Chen, C. C. Xi. Good tilting modules and recollements of derived module categories, II. J. Math. Soc. Japan, 2019, 71(2): 515-554.

\bibitem{CX02}    H. X. Chen and C. C. Xi, Recollements of derived categories, III: Finitistic dimensions, J. London Math. Soc. (2), 2017, 95:633-658.

\bibitem{CX03} H. X. Chen, C. C. Xi. Recollements induced from tilting modules over tame hereditary algebras. Forum Math., 2015, 27(3):1849-1901.

\bibitem{HS} P. J. Hilton, U. Stammbach,  A Course in Homological Algebra, 2nd ed. New York: Springer-Verlag, 2003.

\bibitem{Hap} D. Happel. Reduction techniques for homological conjectures. Tsukuba J. Math., 1993, 17(1):115-130.

\bibitem{HY}    Y. Han. Recollements and Hochschild theory. J. Algebra, 2014, 397:535-547.

\bibitem{HZZ} J. Hu, D. Zhang, P. Zhou. Proper classes and Gorensteinness in extriangulated categories. J. Algebra, 2020, 551: 23-60.

\bibitem{HHZ}J. He, Y.G. Hu, P. Y. Zhou. Torsion pairs and recollements of extriangulated categories. arXiv:2104.04924.

 \bibitem{GMT} W.L. Gu, X. Ma, L. L. Tan.  Homological dimensions of extriangulated categories and recollements. arXiv:2104.06042.

\bibitem{K} S. K\"{o}nig, Tilting complexes, perpendicular categories and recollements of derived  module categories of rings. J. Pure Appl. Algebra, 1991, 73: 211-232.


\bibitem{MV} R. MacPherson, K. Vilonen. Elementary construction of perverse sheaves. Invent. Math., 1986, 84(2): 403-435.
    
\bibitem{MH}X. Ma, Z. Huang. Torsion pairs in recollements of abelian categories. Front. Math. China, 2018, 13(4): 875-892.
    
\bibitem{N}    A. Neeman, Triangulated categories, Ann. of Math. Stud., 148, Princeton Univ. Press, Princeton, NJ, 2001.

\bibitem{N01} P. Nicol\'{a}s. On torsion torsionfree triples. Ph.D. Thesis, Universidad de Murcia (2007).
\bibitem{NOS}   H. Nakaoka, Y. Ogawa, A. Sakai. Localization of extriangulated categories. arXiv:2103.16907.

\bibitem{NP} H. Nakaoka, Y. Palu.  Extriangulated categories, Hovey twin cotorsion pairs and model structures. Cah. Topol. G\'{e}om. Diff\'{e}r. Cat\'{e}g., 2019 60(2): 117-193.

\bibitem{NP1} H. Nakaoka, Y. Palu. External triangulation of the homotopy category of exact quasi-category. arXiv: 2004.02479, 2020.

\bibitem{P01} C. Psaroudakis, J. Vit\'{o}ria. Recollements of Module Categories. Appl. Categor. Struct., 2014, 22:579-593.

\bibitem{P} C. Psaroudakis. Homological theory of recollements of abelian categories. J. Algebra, 2014, 398: 63-110.

\bibitem{P0} B. Parshall, Finite dimensional algebras and algebraic groups. Contemp. Math., 1989, 82: 97-114.

\bibitem{PS} B. Parshall, L. Scott, Derived categories, quasi-hereditary algebras and algebraic groups, Carlton Univ. Math. Notes, 1988, 3:1-104.

\bibitem{QH}    Y.Y. Qin and Y. Han. Reducing homological conjectures by n-recollements. Algebr. Represent. Theory, 2016, 19:377-395.

\bibitem{Rump}   W. Rump, The acyclic closure of an exact category and its triangulation. J. Algebra 2021, 565:402-440.

\bibitem{W} C. A. Weibel An Introduction to Homological Algebra. Cambridge: Cambridge University Press, 1994.

\bibitem{WWZ} L. Wang. J. Wei, H. Zhang. Recollements of extriangulated categories. arXiv: 2012.03258, 2020.


\bibitem{Y}    D. Yang. Recollements from generalized tilting. Proc. Amer. Math. Soc., 2012, 140:83-91.

\bibitem{YG} Y. Q. Yin, N. Gao, Upper recollements and triangle expansions. Acta  Math. Sci. Ser. A Chin. Ed., 2016, 32(11): 1323-1336.


\bibitem{ZZ} P. Zhou, B. Zhu. Triangulated quotient categories revisited. J. Algebra, 2018, 502: 196-232.

\bibitem{ZhZ} B. Zhu, X. Zhuang. Tilting subcategories in extriangulated categories. Front. Math. China, 2020, 15(1): 225-253.














\end{thebibliography}
\end{document}